\documentclass{amsart}
\usepackage{amsmath,amsthm,amsfonts,amssymb,amscd,amsbsy,pgf,tikz,microtype}
\usepackage{caption,enumerate}
\usepackage{dsfont,lscape}

\usepackage[all]{xy}
\usepackage{hyperref}

\usepackage{bm} 

\hypersetup{
    pdftoolbar=true,
    pdfmenubar=true,
    pdffitwindow=false,
    pdfstartview={FitH},
    pdftitle={},
    pdfauthor={},
    pdfsubject={},
    pdfkeywords={}
    pdfnewwindow=true,
    colorlinks=true, 
    linkcolor=blue,
    citecolor=blue,
    urlcolor=black,
}

\newcommand{\dd}{\mathrm{d}}

\newcommand{\pr}{{\rm{pr}}}
\newcommand{\hor}{{\rm{hor}}}
\newcommand{\ver}{{\rm{ver}}}
\newcommand{\odd}{{\rm{odd}}}
\newcommand{\even}{{\rm{even}}}
\newcommand{\triv}{{\rm{triv}}}

\newcommand{\Vol}{\operatorname{Vol}}

\newcommand{\Ric}{\operatorname{Ric}}

\newcommand{\N}{\mathds N}
\newcommand{\Z}{\mathds Z}
\newcommand{\R}{\mathds R}
\newcommand{\C}{\mathds C}

\newcommand{\G}{\mathsf{G}}

\newcommand{\g}{\mathrm g}

\newcommand{\binfty}{\boldsymbol{\infty}} 

\allowdisplaybreaks

\newtheorem{theorem}{Theorem}[]
\newtheorem{lemma}[theorem]{Lemma}
\newtheorem{proposition}[theorem]{Proposition}
\newtheorem{corollary}[theorem]{Corollary}

\newtheorem*{prob*}{\sc Problem}

\newtheorem{mainthm}{\sc Theorem}

\newtheorem*{mainthm*}{\sc Theorem}

\theoremstyle{definition}
\newtheorem{definition}[theorem]{Definition}

\theoremstyle{remark}
\newtheorem{remark}[theorem]{Remark}

\title[Nonplanar minimal spheres in ellipsoids of revolution]{Nonplanar minimal spheres\\ in ellipsoids of revolution}


\author[R. G. Bettiol]{Renato G. Bettiol}
\address[R. G. Bettiol]{\newline
\indent \!\!\!\begin{tabular}{lll}
CUNY Lehman College & & CUNY Graduate Center \\
Department of Mathematics & & Department of Mathematics \\
250 Bedford Park~Blvd W & & 365 Fifth Avenue \\
Bronx, NY, 10468, USA & & New York, NY, 10016, USA
\end{tabular}
}
\email{r.bettiol@lehman.cuny.edu}

\author[P. Piccione]{Paolo Piccione}
\address[P. Piccione]{\newline
\indent Great Bay University \newline
\indent Department of Mathematics \newline
\indent School of Sciences \newline
\indent Dongguan, Guangdong 523000, China}

\address{\emph{Permanent address:} \newline
\indent Universidade de S\~ao Paulo \newline
\indent Departamento de Matem\'atica \newline
\indent Rua do Mat\~ao, 1010 \newline
\indent S\~ao Paulo, SP, 05508-090, Brazil}
\email{piccione@ime.usp.br}

\allowdisplaybreaks
\numberwithin{equation}{section}
\numberwithin{theorem}{section}

\date{\today}

\begin{document}

\begin{abstract}
We use global bifurcation techniques to establish the existence~of arbitrarily many  geometrically distinct nonplanar embedded smooth minimal $2$-spheres in sufficiently elongated $3$-dimensional ellipsoids of revolution. 
More precisely, we quantify the growth rate of the number of such minimal spheres, and describe their asymptotic behavior as the ellipsoids converge to a cylinder.
\end{abstract}

\maketitle

\vspace{-.2cm}
\section{Introduction}
Consider $3$-dimensional ellipsoids in $\R^4$ with semiaxes $a,b,c,d$, given by:
\begin{equation}\label{eq:ellipsoid}
E(a,b,c,d):=\left\{(x_1,x_2,x_3,x_4)\in\R^4: \frac{x_1^2}{a^2}+\frac{x_2^2}{b^2}+\frac{x_3^2}{c^2}+\frac{x_4^2}{d^2}=1\right\}.
\end{equation}
The reflection about a coordinate hyperplane $x_i=0$ is an isometry of $E(a,b,c,d)$, so its fixed point set $\Sigma_i(a):=E(a,b,c,d)\cap \{x_i=0\}$ is a totally geodesic (in particular, minimal) $2$-sphere. 
Henceforth, we refer to these  as \emph{planar} minimal $2$-spheres.

The following problem was proposed by Yau~\cite[p.~127]{yau-prob}:

\begin{prob*}[Yau, 1987]
Are all minimal $2$-spheres in $E(a,b,c,d)$ planar?
\end{prob*}

Let us mention two motivations for this problem. 
First, a well-known theorem of Almgren~\cite{almgren} implies an affirmative answer if $a=b=c=d$. Second, by the solution to Smale's Conjecture~\cite{hatcher,bamler-kleiner}, the space of embedded $2$-spheres in $S^3$ deformation retracts onto $\R P^3$; so, heuristically applying Morse theory to the area functional on this space, one expects 
at least $4$ embedded minimal $2$-spheres in any Riemannian manifold diffeomorphic to $S^3$. (This expectation has been recently proved for bumpy metrics in \cite{wangzhou}.) Under that expectation,
in analogy with $2$-dimensional ellipsoids in $\R^3$ with distinct semiaxes having the least possible number of simple closed geodesics, Yau's problem asks whether $E(a,b,c,d)$ has the least possible number of minimal $2$-spheres, if the semiaxes are all distinct. 

A negative answer to Yau's problem for sufficiently elongated ellipsoids was given by Haslhofer and Ketover~\cite[Thm~1.5]{hk}. Namely, using Min-Max Theory and Mean Curvature Flow, they established the existence of \emph{at least one} nonplanar embedded minimal $2$-sphere in $E(a,b,c,d)$, provided $a$ is sufficiently large, for fixed $b,c,d$.
Our main results refine this negative answer under a symmetry assumption:

\begin{mainthm*} 
If at least two of the semiaxes $b,c,d$ are equal, then there are arbitrarily many geometrically distinct nonplanar embedded minimal $2$-spheres in $E(a,b,c,d)$, provided $a$ is sufficiently large.
\end{mainthm*}

By \emph{geometrically distinct} minimal $2$-spheres we mean they are noncongruent, i.e., cannot be obtained from one another via isometries of $E(a,b,c,d)$. Note that if at least two among $a,b,c,d$ coincide, then any minimal $2$-sphere in $E(a,b,c,d)$ trivially gives rise to infinitely many minimal $2$-spheres, which are pairwise congruent.

\subsection{Statement of main results}
We prove two main results that imply the above Theorem, and describe certain aspects of the ensuing nonplanar embedded minimal $2$-spheres.
Henceforth, we shall assume that $b=c$, the case $c=d$ being totally analogous.
The first main result provides an estimate on the rate in which new nonplanar minimal $2$-spheres appear in $E(a,b,b,d)$ as the parameter $a$ grows:

\begin{mainthm}\label{mainthmA}
The number $N(a)$ of geometrically distinct (up to congruence) nonplanar embedded minimal $2$-spheres in $E(a,b,b,d)$ satisfies
\begin{equation}\label{eq:liminfNa}
\liminf_{a\to+\infty} \frac{ N(a)}{a} \geq \frac{1}{2d}.
\end{equation}
\end{mainthm}

Our second main result gives further geometric information on these nonplanar minimal $2$-spheres and their asymptotic behavior as $a\nearrow+\infty$. Clearly, $E(a,b,b,d)$ converges smoothly to the elliptic cylinder
$E(\infty,b,b,d)=\Sigma_1(\infty)\times \R$ in $\R^4$, where
\begin{equation*}
    \Sigma_1(\infty)=\left\{(0,x_2,x_3,x_4)\in\R^4: \frac{x_2^2}{b^2}+\frac{x_3^2}{b^2}+\frac{x_4^2}{d^2}=1\right\}
\end{equation*}
is the limit of the planar minimal $2$-sphere $\Sigma_1(a)$ as $a\nearrow+\infty$. We prove that the nonplanar minimal $2$-spheres in $E(a,b,b,d)$ from Theorem~\ref{mainthmA} converge smoothly to $\Sigma_1(\infty)$ as $a\nearrow+\infty$, with arbitrarily large multiplicity, area, and Morse index:

\begin{mainthm}\label{mainthmB}
Given $m\geq2$, there exist $a_m>0$ and a nonplanar embedded minimal $2$-sphere $S_m(a)$ in $E(a,b,b,d)$ for all $a > a_m$, which intersects $\Sigma_4(a)$ transversely at $m$ disjoint parallel circles, 
and, as $a\nearrow+\infty$, converges smoothly (away from the singular points $(0,0,0,\pm d)\in\R^4$) to $\Sigma_1(\infty)$ with multiplicity $m$. In particular, their areas converge, i.e., $|S_m(a)|\to m \cdot |\Sigma_1(\infty)|$ as $a\nearrow+\infty$.
Moreover, $(a_m)_{m\geq2}$ is strictly increasing, and, for any sequence $(\varepsilon_m)_{m\geq2}$ with $0<\varepsilon_m\leq a_{m+1}-a_m$, 
\begin{equation}\label{eq:liminf-scarring}
\liminf_{m\to+\infty} \frac{\mathrm{index}(S_{m}(a_m+\varepsilon_m))}{|S_m(a_m+\varepsilon_m)|} \geq \frac{1}{|\Sigma_1(\infty)|}.
\end{equation}
\end{mainthm}

The exact value of $a_m$ could, in principle, be computed by solving two equations that involve (infinite) continued fractions, as explained in Appendix~\ref{appendixA}. While this arithmetic problem seems to be beyond the reach of currently available methods, related numerical experiments lead us to conjecture that $a_m=m$ if $b=c=d=1$.

It is conceivable that
$S_2(a)$ is congruent to the nonplanar minimal $2$-sphere $S_{H\! K}$ found by Haslhofer--Ketover for sufficiently large $a$. Note that $S_{H\! K}\subset E(a,b,b,d)$ converges as a varifold to $\Sigma_1(\infty)$ with multiplicity $2$ as $a\nearrow+\infty$, just like $S_2(a)$, see~\cite[Prop.~1.6]{hk}, but it is unclear to us whether $S_2(a)$ realizes the second width of $E(a,b,b,d)$, as $S_{H\! K}$ does. More generally, for all $m\geq2$, it would be interesting to determine if $S_m(a)$ can be obtained from $m$-parameter sweepouts of~$E(a,b,b,d)$.

A feature of $S_m(a)$ for \emph{even} values of $m\geq2$ is that $S_m^\pm(a):=S_m(a)\cap \{\pm x_1\geq0\}$ are \emph{free boundary} minimal $2$-disks in the ellipsoidal hemispheres $E^\pm(a,b,b,d):=E(a,b,b,d)\cap \{\pm x_1\geq0\}$, see Remark~\ref{rem:freebdy}. In particular, as $a\nearrow+\infty$, these free boundary minimal disks converge smoothly, away from $(0,0,0,\pm d)$, to the boundary $\Sigma_1(\infty)$ with multiplicity $m/2$. Moreover, inequality \eqref{eq:liminf-scarring} remains valid for the ratios $\mathrm{index}(S_m^\pm(a_m+\varepsilon_m))/|S_m^\pm(a_m+\varepsilon_m)|$, where $m\to+\infty$ through even values. This also implies a free boundary version of Theorem~\ref{mainthmA}; namely, the number $N_{D}(a)$ of geometrically distinct nonplanar embedded free boundary minimal $2$-disks in $E^+(a,b,b,d)$ satisfies 
\begin{equation*}
\liminf_{a\to+\infty}\frac{N_D(a)}{a}\geq\frac{1}{4d}.
\end{equation*}

The situation described in Theorem~\ref{mainthmB} bears several analogies with the \emph{scarring} phenomenon recently discovered by Song and Zhu~\cite{scarring}, where stable minimal hypersurfaces $S$ on a generic closed manifold $(M^n,\g)$, $3\leq n\leq 7$, are shown to be the (renormalized) varifold limit of sequences of minimal hypersurfaces $S_m$ in $(M^n,\g)$ with diverging area and Morse index, and $\mathrm{index}(S_m)/|S_m|\to 1/|S|$ as $m\nearrow+\infty$. While there do not exist any stable minimal surfaces in $E(a,b,b,d)$ for $0<a<+\infty$ because it has $\Ric>0$, the limiting cylinder $E(\infty,b,b,d)$ has $\Ric\geq0$, and $\Sigma_1(\infty)\subset E(\infty,b,b,d)$ is stable. In this sense, Theorem~\ref{mainthmB} establishes the existence of 
minimal $2$-spheres in the varying family of Riemannian manifolds $E(a,b,b,d)$
that \emph{scar} on the limiting stable minimal $2$-sphere $\Sigma_1(\infty)$ as $a\nearrow+\infty$. Similar analogies can be drawn to the works of Colding--DeLellis~\cite{cdl} 
and Hass--Norbury--Rubinstein~\cite{hnr}, on the existence of sequences of minimal surfaces with diverging Morse index that accumulate on stable minimal $2$-spheres. Somewhat paradoxically, the \emph{lack of stability} in our geometric setup (which sets it apart from the above works) is simultaneously one of the main challenges to carry out the desired construction, and also one of the key ingredients in our proof.

While we focus exclusively on the ellipsoids \eqref{eq:ellipsoid} throughout this paper, it is a posteriori clear that conclusions similar to those in Theorems~\ref{mainthmA} and \ref{mainthmB} should hold for minimal $2$-spheres embedded in $3$-spheres that are given by the boundary of more general rotationally invariant convex bodies in $\R^4$, as these become elongated in a direction orthogonal to the $2$-planes of revolution.

\subsection{Overview of proofs}
As in other recent applications of Bifurcation Theory to Geometric Analysis (see \cite{bp-notices,spjms} for surveys), we exploit the instability of a degenerating family of highly symmetric solutions to produce our new solutions.
Namely, the Morse index of the planar minimal $2$-sphere $\Sigma_4(a)$ diverges as $a\nearrow+\infty$,\linebreak and each time $a$ crosses an instant $a_m$ where this Morse index jumps, a new nonplanar minimal $2$-sphere $S_m(a)$ \emph{bifurcates} from $\Sigma_4(a)$.
But, while that yields nonplanar solutions for $a$ near $a_m$, this \emph{local result} is not enough to prove Theorems \ref{mainthmA} and~\ref{mainthmB}.

It is in order to promote the above to a \emph{global result} (in the parameter $a$) that we use the symmetry assumption. More precisely, since $b=c$, the ellipsoids $E(a,b,b,d)$ carry a natural isometric $\mathsf{SO}(2)$-action, given by rotations in the $(x_2,x_3)$-plane, in addition to the reflections 
$\tau_1$ and $\tau_4$
about the hyperplanes $x_1=0$ and $x_4=0$, respectively.
By the classical work of Hsiang--Lawson~\cite{hsiang-lawson}, cf.~Theorem~\ref{thm:HL}, the problem of finding $\mathsf{SO}(2)$-invariant minimal $2$-spheres in $E(a,b,b,d)$ reduces to that of finding free boundary geodesics in the orbit space $\Omega_a=E(a,b,b,d)/\mathsf{SO}(2)$, which is topologically a $2$-disk, endowed with an appropriately rescaled Riemannian metric that degenerates on $\partial \Omega_a$. Due to this degeneracy, we supply an ad hoc proof (Theorem~\ref{thm:existence_gammas}) of the existence of geodesics in $\Omega_a$ starting orthogonally from $\partial \Omega_a$, using solutions to the Plateau problem, following an approach inspired by \cite[Lemma 4.1]{hnr}. Among these geodesics, we have those corresponding to the planar minimal $2$-spheres $\Sigma_1(a)$ and $\Sigma_4(a)$, which are respectively denoted $\gamma_\ver$ and $\gamma_\hor$, see Figure~\ref{fig:verhor}. 
Since the reflections $\tau_1$ and $\tau_4$ commute with the $\mathsf{SO}(2)$-action, they descend to isometries of $\Omega_a$ given by reflections about $\gamma_\ver$ and $\gamma_\hor$, respectively.
There are two special types of free boundary geodesics in $\Omega_a$ that are invariant under certain reflections, which we call \emph{even} and \emph{odd} geodesics; namely, those that start orthogonally from $\partial \Omega_a$ and meet $\gamma_\ver$ orthogonally, or at its central point $O=\gamma_\ver\cap\gamma_\hor$, respectively, see Figure~\ref{fig:even-odd}. This setup enables us to define real-valued functions $f_\even(a,s)$ and $f_\odd(a,s)$, where $s\in \left(-\tfrac\pi2,\tfrac\pi2\right)$, whose zeros determine even and odd geodesics in $\Omega_a$, and hence minimal $2$-spheres in $E(a,b,b,d)$, see Proposition~\ref{prop:even_odd}. The fact that $\gamma_\hor$ is trivially an even and odd geodesic 
for all $a>0$ translates to $f_\even(a,0)=f_\odd(a,0)=0$ for all $a>0$, and this is the \emph{trivial branch} of solutions from which we seek bifurcations. 

Sufficient conditions for the existence of local (continuous) bifurcation branches are given by the celebrated result of Crandall--Rabinowitz~\cite{crandall-rabinowitz}, see Theorem~\ref{thm:CR}, involving the linearizations of $f_\even$ and $f_\odd$. These can be computed, up to rescaling, as boundary values of solutions to the Jacobi equation along $\gamma_\hor$, which is derived as a (singular) Sturm--Liouville ODE in Section~\ref{sec:jacobi}. The corresponding spectral problems are then analyzed in Section~\ref{sec:singularSL}, see Figure~\ref{fig:eigenvalues}. Through this analysis, we find a sequence $(a_m)_{m\geq1}$ of values of the parameter $a$ which are bifurcation instants for $f_\even$ if $m$ is even, and for $f_\odd$ if $m$ is odd. 
Counting the zeros of eigenfunctions whose eigenvalue crosses zero at $a=a_m$, we see that the corresponding minimal $2$-spheres intersect $\Sigma_4(a)$ at $m$ orbits of the $\mathsf{SO}(2)$-action. The first bifurcation instant is $a_1=d$, corresponding to the ellipsoid $E(d,b,b,d)$, and its bifurcation branch consists of \emph{planar} minimal $2$-spheres congruent to $\Sigma_1(d)$ and $\Sigma_4(d)$. Aside from this first uninteresting bifurcation, all other bifurcation branches issuing at $a=a_m$, $m\geq 2$, give rise to nonplanar minimal $2$-spheres. Bifurcation branches are pairwise disjoint because the number of intersections with $\Sigma_4(a)$ is locally constant; in particular, they do not reattach to the trivial branch. This rules out one of the possibilities in the dichotomy established by the global bifurcation theorem of Rabinowitz~\cite{rabinowitz}, see Theorem~\ref{thm:rabinowitz}; so all branches must be noncompact and thus persist for all $a>a_m$, see Figure~\ref{fig:branches}. Compactness of the sets of even and odd geodesics in $\Omega_a$, which follows e.g.~ from \cite{choi-schoen}, is used in a crucial way to apply this global bifurcation result.
The remainder of the proof of Theorem~\ref{mainthmA} follows from estimating the asymptotic growth of $(a_m)_{m\geq1}$, as explained in the end of Section~\ref{sec:main}.

The proof of Theorem~\ref{mainthmB} is given in Section~\ref{sec:asymptotic}, analyzing the limiting behavior of geodesics in $\Omega_a$ as $a\nearrow+\infty$. Namely, they converge smoothly to geodesics in the infinite strip $\Omega_\infty=E(\infty,b,b,d)/\mathsf{SO}(2)$, which can be easily described since $\Omega_\infty$ has a nontrivial (constant) Killing vector field. In particular, the only geodesics of $\Omega_\infty$ that intersect (the limit of) $\gamma_\hor$ finitely many times are vertical segments, so geodesics corresponding to the bifurcation branch issuing from $a=a_m$ must converge to $m$ copies of (the limit of) $\gamma_\ver$. 
Moreover, as $a\nearrow+\infty$, these geodesics develop $m-1$ sharp turns near $\partial \Omega_a$, which correspond to $m-1$ catenoidal necks in the minimal $2$-sphere in $E(a,b,b,d)$ near the fixed points $(0,0,0,\pm d)$ of the $\mathsf{SO}(2)$-action. This implies they have Morse index at least $m-1$, leading to \eqref{eq:liminf-scarring}.

\subsection{Acknowledgements}
It is a pleasure to thank Camillo De Lellis, Fernando Cod\'a Marques, 
Davi M\'aximo, and Joaqu\'\i n P\'erez for many discussions on minimal surfaces and the Plateau problem, and Jimmy Petean for conversations on global bifurcation. We also thank John Toland for providing us a copy of his book \cite{toland}, which was an excellent reference for some of the background results that we employ.\newline
\indent The first-named author was supported by the National Science Foundation grant DMS-1904342 and CAREER grant DMS-2142575, by PSC-CUNY Award \# 62074-00 50, and Fapesp (2019/19891-9). The second-named author was supported by grants from CNPq and Fapesp (2022/14254-3, 2022/16097-2), Brazil.

\section{Preliminaries}
For the convenience of the reader, we recall the symmetry reduction principle for $\G$-invariant minimal submanifolds pioneered by Hsiang and Lawson~\cite{hsiang-lawson}, 
and the bifurcation theorems of Crandall and Rabinowitz~\cite{crandall-rabinowitz} and Rabinowitz~\cite{rabinowitz}.

\subsection{Hsiang--Lawson reduction}
Suppose $(M,\g)$ is a complete Riemannian manifold with an isometric action of a compact Lie group $\G$, and let $\Pi\colon M\to M/\G$ be the projection map to the orbit space. It is well-known that the principal part $M_{\pr}\subset M$ is an open, dense, and connected subset, and $(M_{\pr}/\G,\check{\g})$ is a (smooth) Riemannian manifold, such that $\Pi\colon M_{\pr}\to M_{\pr}/\G$ is a Riemannian submersion, see e.g.~\cite[Sec 3.4-3.5]{mybook}. For simplicity, we assume that $M$ has no exceptional orbits, that is, nonprincipal orbits have dimension strictly smaller than principal orbits.

The volume function of principal orbits, $V\colon M_{\pr}/\G\to \R$, $V(x)=\Vol(\Pi^{-1}(x))$, is a smooth function that extends to a continuous function
\begin{equation}\label{eq:vol}
V\colon M/\G\longrightarrow \R,
\end{equation}
which vanishes identically on $\partial(M/\G)$.
The \emph{cohomogeneity} of a $\G$-invariant submanifold $\Sigma\subset M$ is the codimension $k$ that principal $\G$-orbits have inside $\Sigma$. 
For instance, principal orbits are themselves $\G$-invariant submanifolds of $M$, of cohomogeneity $k=0$. 
Clearly, e.g., by Palais' symmetric criticality principle, such a cohomogeneity $0$ submanifold $\Sigma\subset M$ is minimal if and only if $\Sigma/\G=\Pi(\Sigma)\in M_{\pr}/\G$ is a critical point of \eqref{eq:vol}. This characterization is naturally extended to submanifolds of cohomogeneity $k\geq1$ in the following well-known result~\cite[Thm.~2]{hsiang-lawson}.

\begin{theorem}[Hsiang--Lawson]\label{thm:HL}
A $\G$-invariant submanifold $\Sigma$ of cohomogeneity $k\geq1$ is minimal in $(M,\g)$ if and only if the projection $\Sigma_{\pr}/\G=\Pi(\Sigma\cap M_\pr)$ of its principal part is minimal in $(M_{\pr}/\G,V^{2/k}\check{\g})$.
\end{theorem}

\subsection{Bifurcation theory}\label{sub:biftheory}
Although many results stated below hold in far greater generality, see e.g.~\cite{toland, kielhofer}, we focus on the following simple $2$-dimensional bifurcation setup, which suffices to prove Theorems \ref{mainthmA} and \ref{mainthmB} in the Introduction.

Let $f\colon (0,+\infty)\times\mathcal I\to\R$ be a real analytic function, where $\mathcal I\subset\R$ is an open interval with $0\in\mathcal I$. Suppose that for all $a>0$, the
value $s=0$ is a solution of
\begin{equation}\label{eq:bifeq}
f(a,s)=0,
\end{equation}
i.e., $f(a,0)=0$ for all $a>0$. This defines a subset of $f^{-1}(0)\subset (0,+\infty)\times\mathcal I$ denoted 
\begin{equation}\label{eq:btriv}
\mathcal B_{\triv}=\big\{(a,0) :a >0\big\},
\end{equation}
called the \emph{trivial branch} of solutions. 
If the closure of $f^{-1}(0)\setminus \mathcal B_{\triv}$ contains the point $(a_*,0)$, then 
$a_*$ is called a \emph{bifurcation instant} for \eqref{eq:bifeq}, and $(a_*,0)\in \mathcal B_{\triv}$ is called a \emph{bifurcation point} for \eqref{eq:bifeq}.
By the Implicit Function Theorem, a necessary (but not sufficient) condition for $a=a_*$ to be a bifurcation instant is that it is a \emph{degeneracy instant}, i.e., $\frac{\partial f}{\partial s}(a_*,0)=0$. The set of all bifurcation instants is denoted
\begin{equation}\label{eq:bf}
\mathfrak b(f):=\big\{ a_* \in (0,+\infty) : a_* \text{ is a bifurcation instant for } f(a,s)=0\big\}.
\end{equation}
Given $a_*\in\mathfrak b(f)$, the connected component of the closure of $f^{-1}(0)\setminus\mathcal B_{\triv}$ containing $(a_*,0)$ is called the \emph{bifurcation branch} issuing from $(a_*,0)$, and denoted by $\mathcal B_{a_*}$. Moreover, we let $\mathcal B_{a_*}^\pm := \{(a,s)\in\mathcal B_{a_*}:\pm s>0\}$. Note that $\mathcal B^\pm_{a_*}$ need not be connected, since $\mathcal B_{a_*}\cap \mathcal B_\triv$ may contain other bifurcation points besides $(a_*,0)$.

\subsubsection{Local bifurcation}
A sufficient condition for the existence of a bifurcating branch $\mathcal B_{a_*}$ issuing from $(a_*,0)\in \mathcal B_{\triv}$, and a description of its local structure near $(a_*,0)$, are
given by the following celebrated result of Crandall and Rabinowitz \cite[Thm 1.7]{crandall-rabinowitz}, see also \cite[Sec.~I.5]{kielhofer} or \cite[Thm~8.3.1]{toland}.

\begin{theorem}[Crandall--Rabinowitz]\label{thm:CR}
If the instant $a=a_*$ is such that
\begin{enumerate}[\rm (i)]
\item $\dfrac{\partial f}{\partial s}(a_*,0)=0$,\smallskip

\item $\dfrac{\partial^2 f}{\partial a\partial s}(a_*,0)\neq 0$,
\end{enumerate}
then $a_*\in\mathfrak b(f)$, i.e., $a_*$ is a bifurcation instant for $f(a,s)=0$. More precisely, there is an open neighborhood $U$ of $(a_*,0)$ in $(0,+\infty)\times\mathcal I$, and a real analytic curve $(-\varepsilon,\varepsilon)\ni t\mapsto \big(a(t),s(t)\big)\in U$, with $\big(a(0),s(0)\big)=(a_*,0)$ and $s'(0)>0$, such that
\begin{equation}\label{eq:loc_bif}
f^{-1}(0)\cap U=\big\{(a,0)\in U\big\}\cup \big\{ \big(a(t),s(t)\big) : t\in (-\varepsilon,\varepsilon) \big\}.
\end{equation}
\end{theorem}

\begin{remark}\label{rem:s-param}
The first set on the right-hand side of \eqref{eq:loc_bif} is $\mathcal B_\triv\cap U$, and the second is $\mathcal B_{a_*}\cap U$.
Although the original statement of the Crandall--Rabinowitz Theorem only gives $s'(0)\neq 0$, since $\mathcal I$ is $1$-dimensional in our case, we may (and will) impose $s'(0)>0$, which corresponds to orienting the parametrization of $\mathcal B_{a_*}\cap U$ so that $\big(a(t),s(t)\big)\in \mathcal B^+_{a_*}$ if $t\in (0,\varepsilon)$, and $\big(a(t),s(t)\big)\in \mathcal B^-_{a_*}$ if $t\in (-\varepsilon,0)$. Furthermore, by the Implicit Function Theorem, up to replacing $U$ with a smaller open neighborhood of $(a_*,0)$, we may even assume $s(t)=t$ for all $t\in(-\varepsilon,\varepsilon)$ and reparametrize the above real analytic curve in terms of $s$, obtaining 
\begin{equation}\label{eq:loc_bifs}
\mathcal B_{a_*}\cap U=\big\{(a(s),s) : s\in (-\varepsilon,\varepsilon) \big\}.
\end{equation}
\end{remark}

\begin{remark}\label{rem:zeros-evenodd}
If $\mathcal I=(-S,S)$ is symmetric around $0\in\mathcal I$, and, for each $a>0$, the function $\mathcal I\ni s\mapsto f(a,s)$ is either \emph{even} or \emph{odd}, then the set $f^{-1}(0)\subset (0,+\infty)\times \mathcal I$ is invariant under the reflection $(a,s)\mapsto (a,-s)$. In this case, if $a_*$ satisfies the hypotheses of Theorem~\ref{thm:CR}, 
then $\mathcal B_{a_*}\cap U$ is also invariant under this reflection,
since it maps the sets $\mathcal B_{a_*}^\pm$ to one another. In other words, if $\mathcal B_{a_*}\cap U$ is parametrized as in \eqref{eq:loc_bifs}, then $s\mapsto a(s)$ is even, i.e., $a(-s)=a(s)$ for all $s\in (-\varepsilon,\varepsilon)$.
\end{remark}

\subsubsection{Global bifurcation}
Under a suitable properness assumption, local bifurcation branches of the form \eqref{eq:loc_bif} are subject to a dichotomy: they either reattach to the trivial branch $\mathcal B_{\triv}$, or else connect to the boundary of $(0,+\infty)\times\mathcal I$. This is a celebrated global bifurcation result of Rabinowitz \cite[Thm 1.3]{rabinowitz}, see also \cite[Sec.~II.5.2]{kielhofer}. A more detailed statement on the geometry of these branches is given in \cite[Thm~9.1.1]{toland}, which, in our $2$-dimensional setup, yields the following:

\begin{theorem}[Rabinowitz]\label{thm:rabinowitz}
Let $a_*\in\mathfrak b(f)$ be a bifurcation instant for $f(a,s)=0$ that satisfies the hypotheses {\rm (i)} and {\rm (ii)} of Theorem~\ref{thm:CR}. Denote by
\begin{equation}\label{eq:bif-curve}
(-\varepsilon,\varepsilon)\ni t\mapsto\big(a(t),s(t)\big)\in (0,+\infty)\times \mathcal I
\end{equation}
the real analytic curve parametrizing $\mathcal B_{a_*}$ near $(a_*,0)$,  as in \eqref{eq:loc_bif}, and assume that 
\begin{enumerate}[\rm (1)]
\item the map $t\mapsto a(t)$ is not constant,
\smallskip

\item the restriction of the projection $p\colon(0,+\infty)\times\mathcal I\to(0,+\infty)$ onto the first factor to $f^{-1}(0)\subset (0,+\infty)\times\mathcal I$ is a proper map.
\end{enumerate} 
Then, the map \eqref{eq:bif-curve} can be extended to a piecewise real analytic curve
\begin{equation}\label{eq:bif-curve-extended}
\phantom{(-\varepsilon,)}\R\ni t\mapsto \big(a(t),s(t)\big)\in(0,+\infty)\times\mathcal I,
\end{equation}
with $s(t)=0$ if and only if $t=0$, whose restrictions to $(-\infty,0)$ and to $(0,+\infty)$ take values in $\mathcal B^-_{a_*}$ and $\mathcal B^+_{a_*}$, respectively, 
such that, as $t\nearrow+\infty$, either:
\begin{enumerate}[\rm (I)]
\item\label{dich:reattach} the curve \eqref{eq:bif-curve-extended} reattaches to $\mathcal B_\triv$, i.e., $\lim\limits_{t\to+\infty}\big(a(t),s(t)\big)=(a_{**},0)$, where $a_{**}\in\mathfrak b(f)$ is a bifurcation instant distinct from $a_*$;
\smallskip

\item\label{dich:bdy} the curve  \eqref{eq:bif-curve-extended} approaches the boundary of $(0,+\infty)\times\mathcal I$ as $t\nearrow+\infty$, i.e., 
\begin{equation*}
\lim\limits_{t\to+\infty}\big(a(t),s(t)\big)\in(0,+\infty)\times\partial\mathcal I, \;\; \text{ or }
 \;\;  \lim\limits_{t\to+\infty}a(t)=0,  \;\;\text{ or }  \;\; \lim\limits_{t\to+\infty}a(t)=+\infty,
\end{equation*}
\end{enumerate}
and an analogous dichotomy holds as $t\searrow-\infty$.
\end{theorem}

\begin{remark}\label{rem:zeros-evenodd2}
Consider the situation of Remark~\ref{rem:zeros-evenodd}, where the sets $\mathcal B_{a_*}^\pm$ are mapped to one another by $(a,s)\mapsto (a,-s)$. If $a_*\in\mathfrak b(f)$ satisfies the hypotheses of Theorem~\ref{thm:rabinowitz}, then \eqref{eq:bif-curve-extended} can be chosen so that $\big(a(-t),s(-t)\big)=\big(a(t),-s(t)\big)$ for all $t\in\R$; in particular, the \emph{same} among \eqref{dich:reattach} or \eqref{dich:bdy} holds as $t\nearrow+\infty$ \emph{and} as $t\searrow-\infty$.
\end{remark}

\subsubsection{Disjoint branches}
We now present an abstract sufficient condition to ensure that different bifurcation branches do not intersect. This approach is inspired by similar results for Yamabe-type PDEs in \cite{yanyan,petean}, where global branches are distinguished by the nodal properties of their solutions, see also \cite[Sec.~III.6]{kielhofer}. 

\begin{definition}\label{def:invariant}
A \emph{discrete-valued invariant} for the equation $f(a,s)=0$ 
is a locally constant function $Z\colon f^{-1}(0)\setminus\mathcal B_\triv \to \N_0$.
\end{definition}

In the above, and throughout, we denote by $\N_0=\N\cup\{0\}$ the set of nonnegative integers.
Clearly, if $Z$ is a discrete-valued invariant for $f(a,s)=0$ and $a_*\in\mathfrak b(f)$, then $Z$ is constant on each connected component of $\mathcal B_{a_*}\setminus\mathcal B_\triv$.
Moreover, if the hypotheses of Theorem~\ref{thm:CR} hold at $(a_*,0)$, then in light of the local form \eqref{eq:loc_bif}, there exists an open neighborhood $\mathcal O$ of $(a_*,0)$ in $(0,+\infty)\times\mathcal I$, such that $\mathcal O\cap \mathcal B^+_{a_*}$ and $\mathcal O\cap \mathcal B^-_{a_*}$ are connected, so $Z$ is constant on each of $\mathcal O\cap \mathcal B^\pm_{a_*}$, and we write:
\begin{equation}\label{eq:zpm}
z^\pm(a_*):=Z|_{\mathcal O\cap \mathcal B^\pm_{a_*}}.
\end{equation}

\begin{proposition}\label{prop:disjoint_noncompact}
Let $Z$ be a discrete-valued invariant for $f(a,s)=0$, and assume that 
the restriction to $f^{-1}(0)$ of the projection $p\colon(0,+\infty)\times\mathcal I\to(0,+\infty)$ onto the first factor is a proper map. Suppose the hypotheses of Theorem~\ref{thm:CR} and hypothesis {\rm (1)} of Theorem~\ref{thm:rabinowitz} hold at every $a_*\in\mathfrak b(f)$.
If for all $a_1,a_2\in \mathfrak b(f)$, we have that $z^+(a_1)\neq z^+(a_2)$, then 
the sets $\overline{\mathcal B_{a_*}^+}$, $a_*\in \mathfrak b(f)$, are noncompact and pairwise disjoint; and similarly for $\overline{\mathcal B_{a_*}^-}$, if $z^-(a_1)\neq z^-(a_2)$ for all $a_1,a_2\in\mathfrak b(f)$.
\end{proposition}

\begin{proof}
Since Theorem~\ref{thm:CR} holds at every $a_*\in\mathfrak b(f)$, the sets $\mathcal B_{a_*}^\pm$ are connected, and their closure is given by $\overline{\mathcal B_{a_*}^\pm}=\mathcal B_{a_*}^\pm \cup \big(\mathcal B_{a_*}\cap\mathcal B_\triv\big)$. In particular, $Z$ is constant and equal to $z^\pm(a_*)$ along each $\mathcal B_{a_*}^\pm$.

As $z^+\colon \mathfrak b(f)\to\N_0$ is assumed injective, it follows that the sets $\mathcal B^+_{a_*}$, $a_*\in\mathfrak b(f)$, are pairwise disjoint. Again from injectivity of $z^+$, we have $\overline{\mathcal B_{a_*}^+}=\mathcal B_{a_*}^\pm\cup\{(a_*,0)\}$, which implies that $\overline{\mathcal B_{a_*}^+}$, $a_*\in\mathfrak b(f)$, are pairwise disjoint. Furthermore, each of them is noncompact by Theorem~\ref{thm:rabinowitz}, as the occurrence of \eqref{dich:reattach} is ruled out by the fact that each $\overline{\mathcal B_{a_*}^+}$ contains only one bifurcation instant.
Similarly, if $z^-\colon \mathfrak b(f)\to\N_0$ is injective, then the same arguments above lead to analogous conclusions for $\overline{\mathcal B_{a_*}^-}$.
\end{proof}

\section{\texorpdfstring{Minimal $2$-spheres of revolution}{Minimal spheres of revolution}}
Let $E(a,b,c,d)$ be an ellipsoid as in \eqref{eq:ellipsoid}, such that at least two among $b, c, d$ are equal. Without loss of generality, to simplify the exposition, we henceforth assume $b=c$ in the remainder of the paper, and denote by $(S^3,\g_a)$ the ellipsoid $E(a,b,b,d)$ with the Riemannian metric induced by the Euclidean metric in $\R^4$.

\subsection{Symmetry reduction}
Consider the isometric action of $\G=\mathsf{SO}(2)$ on $(S^3,\g_a)$ by rotations on the $(x_2,x_3)$-plane, whose orbit through $x=(x_1,x_2,x_3,x_4)\in S^3$ is
\begin{equation}\label{g-orbits}
\G(x)=\{(x_1,\,x_2\cos\theta -x_3\sin\theta,\, x_2\sin\theta+ x_3\cos\theta,\, x_4)\in S^3 : \theta\in\R\}.
\end{equation}
Clearly, $x\in S^3$ is fixed by the action if and only if $x_2=x_3=0$, so the fixed point set $(S^3)^\G$ is an ellipse with semiaxes $a$ and $d$, which is a geodesic in $(S^3,\g_a)$.
All other points belong to principal orbits, i.e., $S^3_{\pr} = S^3\setminus (S^3)^\G$. The orbit space
\begin{equation}\label{eq:quotient}
S^3/\G = \left\{ (x_1,r,x_4)\in\R^3 : \frac{x_1^2}{a^2}+\frac{r^2}{b^2}+\frac{x_4^2}{d^2}=1, \, r\geq 0 \right\}
\end{equation}
has boundary $\partial(S^3/\G)=\{(x_1,0,x_4)\in S^3/\G\}=(S^3)^\G$, and the projection map $\Pi\colon S^3\to S^3/\G$ is given by $\Pi(x)=\big(x_1,\sqrt{x_2^2+x_3^2},\,x_4\big)$. Moreover, the quotient metric $\check{\g}_a$ on $S^3_{\pr}/\G=\{(x_1,r,x_4)\in S^3/\G:r>0\}$ such that $\Pi\colon S^3_{\pr}\to S^3_{\pr}/\G$ is a Riemannian submersion is the metric induced on \eqref{eq:quotient} by the Euclidean metric in~$\R^3$, and the orbital volume function \eqref{eq:vol} is given by
\begin{equation}\label{eq:vol-orbits}
V\colon S^3/\G\longrightarrow\R, \quad V(x_1,r,x_4)=2\pi r=2\pi b\sqrt{1-\frac{x_1^2}{a^2}-\frac{x_4^2}{d^2}}.
\end{equation}

Consider the (open) Riemannian $2$-disk endowed with the conformal metric
\begin{equation*}
\Omega_a:=\big(S^3_{\pr}/\G,V^2\, \check{\g}_a\big),
\end{equation*}
and identify $\partial \Omega_a=\partial(S^3/\G)$.
Observe that the length of a curve $\gamma$ in $\Omega_a$ is equal to the area of $\Pi^{-1}(\gamma)$ in $(S^3,\g_a)$. 
The following is a direct consequence of Theorem~\ref{thm:HL}.

\begin{proposition}\label{prop:reduction}
A $\G$-invariant surface $\Sigma$ in $(S^3,\g_a)$ is minimal if and only if the curve $\Sigma_{\pr}/\G=\Pi(\Sigma\cap S^3_{\pr})$ is a geodesic in $\Omega_a$. 
\end{proposition}

Therefore, $\Sigma$ is a minimal \emph{$2$-sphere} of revolution in $(S^3,\g_a)$ if and only if $\Sigma_{\pr}/\G$ is a geodesic with (limiting) endpoints in $\partial\Omega_a$. 

\begin{definition}\label{def:transv_orth}
A curve $\gamma\colon (t_0,t_1)\to \Omega_a$ with $\lim_{t\searrow t_0} \gamma(t)=x\in \partial \Omega_a$ is called \emph{transverse}, respectively \emph{orthogonal}, to $\partial\Omega_a$ at the 
endpoint $x$ if the limit as $t\searrow t_0$  of $\gamma'(t)/\|\gamma'(t)\|_{\check{\g}_a}$ exists and is transverse, respectively $\check{\g}_a$-orthogonal, to $T_x \partial\Omega_a$.
\end{definition}

\begin{remark}\label{rem:reduction}
A minimal $2$-sphere of revolution $\Sigma$ in $(S^3,\g_a)$ is smooth if and only if the corresponding geodesic is orthogonal to $\partial\Omega_a$ at both endpoints; we call such curves \emph{free boundary geodesics} in $\Omega_a$. Moreover, $\Sigma$ is embedded if and only if the corresponding free boundary geodesic has no self-intersections.
\end{remark}

\begin{definition}\label{def:taus}
The reflections $\tau_1$ and $\tau_4$ of $\R^4$, about the hyperplanes $x_1=0$ and $x_4=0$, are isometries of $(S^3,\g_a)$ that commute with the $\G$-action and hence descend to isometries of $\Omega_a$, that we call $\tau_\ver$ and $\tau_\hor$, respectively. The fixed point sets of $\tau_1$ and $\tau_4$ in $(S^3,\g_a)$ are the totally geodesic planar $2$-spheres
$\Sigma_1(a)$ and $\Sigma_4(a)$,
that project to the fixed point sets of $\tau_\ver$ and $\tau_\hor$ in $\Omega_a$; these are free boundary geodesics in $\Omega_a$ that we denote by $\gamma_\ver$ and $\gamma_\hor$, respectively. For convenience, we indiscriminately use the symbols $\gamma_\ver$ and $\gamma_\hor$ for the maps into $\Omega_a$ and their image. The intersection point of $\gamma_\ver$ and $\gamma_\hor$ is denoted $O:=(0,b,0)\in\Omega_a$. 
\end{definition}

\begin{remark}\label{rem:onlyplanars}
If $a\neq d$, then $\gamma_\ver$ and $\gamma_\hor$ are the only free boundary geodesics in $\Omega_a$ that correspond to \emph{planar} minimal $2$-spheres of revolution in $(S^3,\g_a)$. 
If $a=d$, rotations in the $(x_1,x_4)$-plane induce isometries on $\Omega_a$, so there is a continuous family of free boundary geodesics on $\Omega_a$ that correspond to (pairwise congruent) planar minimal $2$-spheres.
\end{remark}

\subsection{\texorpdfstring{Orthogonal geodesics}{Orthogonal geodesics}}
Since the Riemannian metric of $\Omega_a$ degenerates on $\partial\Omega_a$, existence and uniqueness of geodesics starting orthogonal to $\partial\Omega_a$ need to be properly justified. Recall that $\partial\Omega_a=\partial (S^3/\G)$, and define
\begin{equation}\label{eq:beta}
\beta\colon \R\longrightarrow\partial\Omega_a, \quad \beta(s)=(a\cos s,0,d\sin s).
\end{equation}
Let $\vec v_{a,s}=(0,1,0)\in T_{\beta(s)} (S^3/\G)$ be the orthogonal direction to $\partial \Omega_a$ at $\beta(s)$.

The free boundary geodesics $\gamma_\ver$ and $\gamma_\hor$ are trivial solutions to the initial value problem for orthogonal geodesics starting at $\beta\left(\frac{k\pi}{2}\right)$, $k\in\Z$, see Figure~\ref{fig:verhor}. We now use well-known facts about the Plateau problem to construct orthogonal geodesics starting at any $\beta(s)$, $s\in\R$, following an approach inspired by \cite[Lemma~4.1]{hnr}.
\begin{center}
\vspace{-.3cm}
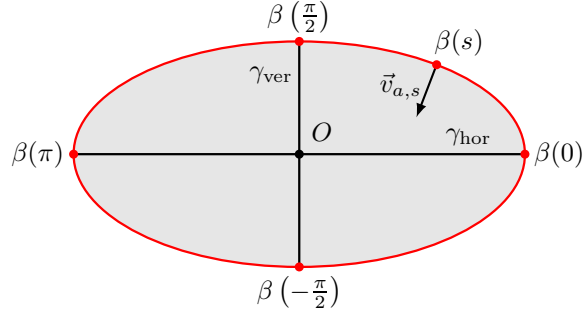
\begin{figure}[!ht]
\begin{tikzpicture}[scale=1.5]
    \filldraw[red,fill=gray!20!white,line width=0.3mm] (0,0) ellipse (2cm and 1cm);
    \draw[->,thick,-latex] (1.21752,0.793353) -- (1.03842, 0.326532);
    \draw[black, line width=0.3mm] (-2,0) -- (2,0);
    \draw[black, line width=0.3mm] (0,-1) -- (0,1);
    \draw (1.5,.15) node {${\gamma_{\hor}}$};
    \draw (-.25,.7) node {${\gamma_{\ver}}$};
    \draw (1.21752+.2,0.793353+.2) node {$\beta(s)$};
    \draw (.9,.6) node {$\vec v_{a,s}$};
    \draw (.2,.2) node {$O$};
    \draw (2,0) node [right] {$\beta(0)$};
    \draw (0,1) node [above] {$\beta\left(\frac{\pi}{2}\right)$};
    \draw (-2,0) node [left] {$\beta(\pi)$};
    \draw (0,-1) node [below] {$\beta\left(-\frac{\pi}{2}\right)$};
    \fill[color=red] (1.21752,0.793353) circle (0.04);
    \fill[color=red] (2,0) circle (0.04);
    \fill[color=red] (0,1) circle (0.04);
    \fill[color=red] (-2,0) circle (0.04);
    \fill[color=red] (0,-1) circle (0.04);
    \fill[color=black] (0,0) circle (0.04);
\end{tikzpicture} 
\caption{Schematic depiction of $\Omega_a$, 
with boundary in red, parametrized by \eqref{eq:beta}, and free boundary geodesics $\gamma_\ver$ and $\gamma_\hor$, with endpoints $\beta\left(\frac{\pi}{2}\right)$ and $\beta\left(-\frac{\pi}{2}\right)$, respectively $\beta(0)$ and $\beta(\pi)$. }\label{fig:verhor}
\end{figure}
\vspace{-.65cm}
\end{center}
\begin{theorem}\label{thm:existence_gammas}
For all $s\in\R$, there is a unique (up to reparametrization) maximal geodesic $\gamma_{a,s}$ in $\Omega_a$ starting transverse to $\partial\Omega_a$ at $\beta(s)$. Moreover, $\gamma_{a,s}$ is orthogonal to $\partial\Omega_a$ at $\beta(s)$, i.e., tangent to $\vec v_{a,s}$, and its dependence on $(a,s)$ is real analytic.
\end{theorem}

\begin{proof}
Let $p\in \Omega_a$ be sufficiently close to $\partial \Omega_a$, so that the $\G$-orbit $\Pi^{-1}(p)$ is an \emph{extremal} curve in $(S^3,\g_a)$, i.e., lies in the boundary of a convex set. Recall that $\Pi^{-1}(p)$ is a circle, which is real analytic and null-homotopic. Thus,
the (area-minimizing) disk solution $D_p$ to the Plateau problem in $(S^3,\g_a)$ with contour $\partial D_p=\Pi^{-1}(p)$ exists, is embedded and smooth up to the boundary, and unique, see e.g.~\cite{white-notes}. In particular, $D_p$ is $\G$-invariant, and hence must contain a fixed point $x\in D_p\cap (S^3)^\G$. Let $s\in \R$ be so that $\Pi(x)=\beta(s)$, and $\gamma_{a,s}\colon (0,\varepsilon)\to\Omega_a$ parametrize the curve $\Pi(D_p\setminus\{x\})\subset\Omega_a$, with $\lim_{t\searrow0}\gamma_{a,s}(t)=\beta(s)$. By Proposition~\ref{prop:reduction}, $\gamma_{a,s}$ is a geodesic. Moreover, the isotropy $\G$-representation on $T_x S^3$ splits as a direct sum of irreducibles $T_x D_p$ and $T_x (S^3)^\G$, which are hence $\g_a$-orthogonal. 
The $\G$-orbit space of the unit sphere in $T_x D_p=(T_x (S^3)^\G)^\perp$ is the $\check{\g}_a$-normal direction to $(S^3)^\G=\partial\Omega_a$ at $\beta(s)$, and hence $\gamma_{a,s}$ is orthogonal to $\partial\Omega_a$.

The preimage under $\Pi$ of any geodesic in $\Omega_a$ that starts transverse to $\partial\Omega_a$ at $\beta(s)$ is a properly embedded minimal punctured disk in $(S^3\setminus\{x\},\g_a)$. By standard removable singularity results, see e.g.~\cite[Prop~1]{choi-schoen}, this minimal punctured disk extends to a smooth minimal disk tangent to $D_p$ at $x$, which must hence agree with $D_p$ by the maximum principle. Therefore, up to reparametrization, the maximal extension of $\gamma_{a,s}$ is the unique geodesic in $\Omega_a$ with the above properties, where we identify $\gamma_{a,s}$ and $\gamma_{a,s+2\pi k}$ for all $k\in\Z$, since \eqref{eq:beta} is $2\pi$-periodic.

The real analytic dependence of $\gamma_{a,s}$ on $(a,s)$ is a consequence of real analytic dependence of the area-minimizing disk $D_p$ on the contour $\Pi^{-1}(p)$, which follows from results of White~\cite[Thm.~3.1]{white87} and \cite{white2}. The nondegeneracy assumption in \cite[Thm.~3.1]{white87}, i.e., the absence of nontrivial Jacobi fields vanishing on the boundary of the minimal disk, is satisfied if $p$ is sufficiently close to $\partial\Omega_a$, which corresponds to $\Pi^{-1}(p)$ being sufficiently small. Thus, moving $p$ around a collar neighborhood of $\partial\Omega_a$, we see this construction covers all points $\beta(s)$, $s\in\R$.
\end{proof}

\begin{remark}
Existence of orthogonal geodesics on certain $2$-disks $\Omega=M_\pr/\G$ with a Riemannian metric $V^{2/k}\check \g$, as in Theorem~\ref{thm:HL}, is proven in Hsiang--Hsiang~\cite{hsiang-hsiang} and Hsiang~\cite{hsiang2,hsiang1}, among others, by reducing the geodesic problem from an ODE system consisting of two coupled equations to a single ODE. However, such a reduction is not possible in Theorem~\ref{thm:existence_gammas}, because of two key differences.
Namely, the $2$-disks $M_\pr/\G$ in the above references have a nontrivial Killing field, since $\G$ has dimension strictly smaller than that of the full isometry group of $(M,\g)$,  and
 $V\colon\Omega\to\R$ is constant on levelsets of the distance function to $\partial \Omega$. Neither of these conditions hold in the setting of Theorem~\ref{thm:existence_gammas}, except for the special case $a=d$.
\end{remark}

\begin{remark}\label{rem:symm}
In addition to the $2\pi$-periodicity, other symmetries arise from the uniqueness statement in Theorem~\ref{thm:existence_gammas} via the reflections $\tau_\ver$ and $\tau_\hor$ across $\gamma_\ver$ and $\gamma_\hor$. Namely, up to reparametrization, $\gamma_{a,s}=\gamma_{a,{-s}}$ and $\gamma_{a,s}=\gamma_{a,{\pi-s}}$. Moreover, $\gamma_{a,\pm\frac{\pi}{2}}$ clearly coincide with $\gamma_\ver$, while $\gamma_{a,0}$ and $\gamma_{a,\pi}$ coincide with $\gamma_\hor$.
\end{remark}

Due to the above symmetries, we shall only consider $\gamma_{a,s}$ where $s$ is in the interval
\begin{equation*}
\mathcal I:=\left(-\tfrac{\pi}{2},\tfrac{\pi}{2}\right).
\end{equation*}
It is convenient to parametrize $\gamma_{a,s}$ by arclength with respect to the quotient metric $\check{\g}_a$, instead of the (conformal) metric $V^2\, \check{\g}_a$ of $\Omega_a$. We denote by $\rho$ this $\check{\g}_a$-arclength parameter, and by $(0,\ell_{a,s})$ the maximal domain of definition for $\gamma_{a,s}(\rho)$.

\begin{proposition}\label{prop:arrival_time}
All geodesics $\gamma_{a,s}$, with $s\in\mathcal I$, intersect $\gamma_\ver$ transversely. In particular, there exists a real analytic function $T_a\colon\mathcal I\to (0,\ell_{a,s})$ such that  the first intersection point is $\gamma_{a,s}(T_a(s))\in \gamma_\ver$, and $\gamma_{a,s}'(T_a(s))$ is transverse to~$\gamma_\ver$.
\end{proposition}

\begin{proof}
Let $F\subset \mathcal I$ be a closed subinterval with $0\in F$, and let $E$ be the subset of $s\in F$ such that $\gamma_{a,s}$ intersects $\gamma_\ver$ transversely. Clearly, $E$ is nonempty as $0\in E$, and open in $F$ by continuity of $s\mapsto \gamma_{a,s}$, see Theorem~\ref{thm:existence_gammas}. Moreover, we claim that $E$ is closed in $F$, and hence $E=F$. Indeed, if $s_n\in E$ is a sequence with $s_n\to s_\infty\in F$ but $s_\infty\notin E$, then $\gamma_{s_\infty,a}$ intersects $\gamma_\ver$ but not transversely, hence tangentially. Such a tangential intersection cannot occur in the interior of $\Omega_a$, as it would contradict uniqueness of geodesics with the same initial condition, and nor at $\partial\Omega_a$, as it would contradict the uniqueness statement in Theorem~\ref{thm:existence_gammas}. Thus $s_\infty\in E$, proving the claim.
Since $F$ is an arbitrary closed subinterval of $\mathcal I$, every $\gamma_{a,s}$, $s\in\mathcal I$, intersects $\gamma_\ver$ transversely. The existence and regularity of the function $T_a\colon\mathcal I\to (0,\ell_{a,s})$ as stated now follows from the Implicit Function Theorem.
\end{proof}

\begin{remark}
If we assume that $\gamma_{a,s}\colon (0,\ell_{a,s})\to\Omega_a$ has no self-intersections and the limit of $\gamma_{a,s}(\rho)$ as $\rho\nearrow\ell_{a,s}$ exists, then the claim that $\gamma_{a,s}$ intersects $\gamma_\ver$ transversely follows from Frankel's Theorem. Indeed, by maximality of $\ell_{a,s}$, the geodesic $\gamma_{a,s}(\rho)$ converges to a point of $\partial\Omega_a$ as $\rho\nearrow\ell_{a,s}$, and arrives to $\partial\Omega_a$ transversely, hence orthogonally by Theorem~\ref{thm:existence_gammas}.
By Remark~\ref{rem:reduction}, the preimage under $\Pi$ of any two free boundary geodesics without self-intersections (such as $\gamma_\ver$, or $\gamma_\hor$, and this putative $\gamma_{a,s}$) are embedded minimal $2$-spheres in the positively curved manifold $(S^3,\g_a)$, which hence intersect. They do so along principal $\G$-orbits and transversely (by the maximum principle), so their images also intersect transversely in $\Omega_a$.
\end{remark}

\begin{proposition}\label{prop:embedded}
For all $s\in\mathcal I$, the restriction of $\gamma_{a,s}\colon (0,\ell_{a,s})\to\Omega_a$ to the interval $(0,T_a(s)]\subset (0,\ell_{a,s})$ has no self-intersections.
\end{proposition}

\begin{proof}
Let $S$ be the subset of $s\in \mathcal I$ such that the restriction of $\gamma_{a,s}$ to $(0,T_a(s)]$ \emph{has} a self-intersection. 
Any such self-intersections must occur inside $\Omega_a$, for otherwise $\gamma_{a,s}$ would have returned to $\partial\Omega_a$ before intersecting $\gamma_\ver$, contradicting maximality of $\gamma_{a,s}$ and Proposition~\ref{prop:arrival_time}. Thus, these self-intersections are transverse, and hence stable under small perturbations, so $S\subset \mathcal I$ is open by Theorem~\ref{thm:existence_gammas}. Moreover, we claim $S\subset\mathcal I$ is closed. Indeed, suppose $s_n\in S$ is a sequence with $s_n\to s_\infty$, and let $\rho_n<\rho'_n$ be sequences such that $\gamma_{a,s_n}(\rho_n)=\gamma_{a,s_n}(\rho'_n)$. Up to reparametrizing $\gamma_{a,s_n}$, assume that $\rho_n\to \rho_\infty$ and $\rho'_n\to \rho'_\infty$.
One cannot have $\rho_\infty=\rho'_\infty=0$, since any self-intersections must occur inside $\Omega_a$; nor $\rho_\infty=\rho'_\infty>0$, as that would imply the existence of arbitrarily short geodesic loops in the Riemannian $2$-disk $\Omega_a$. Thus, $\rho_\infty < \rho'_\infty$, i.e., $s_\infty\in S$, proving the above claim. Therefore, $S\subset\mathcal I$ is open and closed, so either $S=\mathcal I$ or $S=\emptyset$. The geodesic $\gamma_{a,0}=\gamma_\hor$ has no self-intersections in $(0,T_a(0)]$, hence $0\notin S$, so $S=\emptyset$.
\end{proof}

\subsection{Even and odd geodesics}\label{subsec:even-odd}
For $a>0$, $s\in\mathcal I$, and $z\in [0,1)$, let
\begin{equation}\label{eq:sigmas}
\begin{aligned}
\sigma(a,s,z) &:=\phantom{:} \gamma_{a,s}\big(t_a(z)\,T_a(s)\big), \\
\sigma(a,s,z) &\phantom{:}=\phantom{:} \big(\sigma_{x_1}(a,s,z),\, \sigma_r(a,s,z),\, \sigma_{x_4}(a,s,z)\big)\in \Omega_a,
\end{aligned}
\end{equation}
be the reparametrization of $\gamma_{a,s}\colon (0,\ell_{a,s})\to\Omega_a$, where $T_a\colon \mathcal I\to (0,\ell_{a,s})$ is given by Proposition~\ref{prop:arrival_time}, and $t_a\colon [0,1)\to (0,1]$ is the decreasing real analytic function so that $\gamma_\hor$, parametrized as $z\mapsto\gamma_{a,0}(t_a(z)T_a(0))$, has $x_1$-coordinate affine in $z$.

In other words, the above choices defining \eqref{eq:sigmas} are such that 
$\sigma(a,s,0)\in\gamma_\ver$ for all $a>0$ and $s\in\mathcal I$, and 
$\sigma_{x_1}(a,0,z)=az$ for all $a>0$ and $z\in [0,1)$. Note that, from Theorem~\ref{thm:existence_gammas} and Proposition~\ref{prop:arrival_time}, the map \eqref{eq:sigmas} is real analytic.

\begin{definition}\label{def:fevenfodd}
Using the notation above, define the real analytic functions
\begin{equation*}
\begin{aligned}
&f_{\even}\colon (0,+\infty)\times \mathcal I \longrightarrow\R,
& \qquad &f_{\odd}\colon (0,+\infty)\times \mathcal I \longrightarrow\R, \\
&f_\even(a,s) = \tfrac{\partial }{\partial z}\sigma_{x_4}(a,s,z)\big|_{z=0}, & 
 &f_\odd(a,s) = \sigma_{x_4}(a,s,0).
\end{aligned}
\end{equation*}
The geodesic $\gamma_{a,s}\colon (0,\ell_{a,s})\to\Omega_a$ is called \emph{even}
if $f_\even(a,s)=0$, and it is called \emph{odd} if $f_\odd(a,s)=0$.
\end{definition}

Note that the geodesic $\gamma_{a,s}$ is even if and only if it intersects $\gamma_\ver$ orthogonally, and odd if and only if it intersects $\gamma_\ver$ at the point $O$, see Figure~\ref{fig:even-odd}.

\begin{remark}\label{rem:feven-even_fodd-odd}
It follows from Remark~\ref{rem:symm} that, for fixed $a>0$, the functions $f_\even$ and $f_\odd$ are even and odd, respectively, as functions of $s\in\mathcal I$, i.e.,
\begin{equation}\label{eq:evenfeven-oddfodd}
f_\even(a,-s)=f_\even(a,s) \quad\text{ and }\quad f_\odd(a,-s)=-f_\odd(a,s).
\end{equation}
In particular, $f^{-1}(0)\subset (0,+\infty)\times \mathcal I$, for both $f=f_\even$ and $f_\odd$, are invariant under the reflection $(a,s)\mapsto(a,-s)$, see also Remarks~\ref{rem:zeros-evenodd} and \ref{rem:zeros-evenodd2}.
\end{remark}

\begin{proposition}\label{prop:even_odd}
If the geodesic $\gamma_{a,s}\colon (0,\ell_{a,s})\to\Omega_a$, $s\in\mathcal I$, is either even or odd, 
then $\Pi^{-1}(\gamma_{a,s})$ is a smooth embedded minimal $2$-sphere in $(S^3,\g_a)$. If $a\neq d$, then this minimal $2$-sphere is planar if and only if $s=0$.
\end{proposition}

\begin{proof}
If $\gamma_{a,s}$ is even, then it intersects $\gamma_\ver$ orthogonally, hence is mapped to itself by the reflection $\tau_\ver$. Similarly, if $\gamma_{a,s}$ is odd, then it intersects $\gamma_\ver$ at $O$, and hence is mapped to itself by the isometry $\tau_\ver\circ\tau_\hor$.
In both cases, $\gamma_{a,s}$ is a free boundary geodesic in $\Omega_a$, i.e., meets $\partial\Omega_a$ orthogonally at both endpoints, and has no self-intersections by Proposition~\ref{prop:embedded}. The stated conclusions now follow from Proposition~\ref{prop:reduction}, see also Remarks~\ref{rem:reduction} and \ref{rem:onlyplanars}.
\end{proof}

\vspace{-.2cm}
\begin{center}
\begin{figure}[!ht]
\begin{tikzpicture}[scale=1.5]
    \filldraw[red,fill=gray!20!white,line width=0.3mm] (0,0) ellipse (2cm and 1cm);
    \draw[black, line width=0.3mm] (-2,0) -- (2,0);
    \draw[black, line width=0.3mm] (0,-1) -- (0,1);
    \draw (-.25,.7) node {${\gamma_{\ver}}$};
    \draw[-,black!40!green,thick] (0,0) to[out=-45,in=-111] (1.21752,0.793353);
    \draw[-,black!40!green,thick] (0,0) to[out=135,in=69] (-1.21752,-0.793353);
    \draw[-,blue,thick] (0,-.4) to[out=0,in=-99.51] (0.635961, 0.948097);
    \draw[-,blue,thick] (0,-.4) to[out=180,in=-80.48] (-0.635961, 0.948097);
    \draw (.2,.2) node {$O$};
    \fill[color=black] (0,0) circle (0.04);
    \draw[red,line width=0.3mm] (0,0) ellipse (2cm and 1cm);
\end{tikzpicture}
\caption{Schematic representation of an even geodesic (blue), and an odd geodesic (green) in $\Omega_a$.}\label{fig:even-odd}
\end{figure}
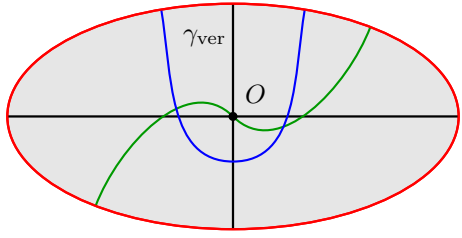
\end{center}
\vspace{-.7cm}

\begin{remark}\label{rem:freebdy}
Let $(S^3_+,\g_a)$ be the ellipsoidal hemisphere consisting of the points of $(S^3,\g_a)$ with $x_1\geq 0$. Clearly, $(S^3_+,\g_a)$ has totally geodesic boundary $\Sigma_1(a)$. Since an even geodesic $\gamma_{a,s}$ intersects $\gamma_\ver$ orthogonally at some $p\in\Omega_a$, the corresponding minimal $2$-sphere $\Pi^{-1}(\gamma_{a,s})$ also intersects $\Pi^{-1}(\gamma_\ver)=\Sigma_1(a)$ \emph{orthogonally} along the $\G$-orbit of $\Pi^{-1}(p)$ in $(S^3,\g_a)$. Thus, $\Pi^{-1}(\gamma_{a,s})\cap S^3_+$ is a \emph{free boundary minimal $2$-disk} in the ellipsoidal hemisphere $(S^3_+,\g_a)$.
\end{remark}

Since $\sigma(a,0,z)$, $z\in[0,1)$, parametrizes the horizontal geodesic $\gamma_\hor$, we have that $\sigma_{x_4}(a,0,z)\equiv 0$. In other words, $\gamma_{a,0}$ is (the only geodesic $\gamma_{a,s}$, $s\in\mathcal I$, to be) simultaneously even and odd, i.e.,
\begin{equation}\label{eq:trivial_branch}
f_\even(a,0)=f_\odd(a,0)=0, \qquad \text{ for all } a>0.
\end{equation}
Thus, we have a trivial branch of solutions $\mathcal B_{\triv}$, exactly as in \eqref{eq:btriv}, to the equation $f(a,s)=0$, with $f=f_\even$ or $f_\odd$. 
By finding solutions $(a,s)\in (0,+\infty)\times \mathcal I$ that bifurcate from $\mathcal B_{\triv}$ as $a\nearrow+\infty$, we shall find even and odd geodesics $\gamma_{a,s}$, $s\neq0$, and hence nonplanar minimal $2$-spheres in $(S^3,\g_a)$ by Proposition~\ref{prop:even_odd}.

\section{Jacobi equation of the trivial (planar) solution}\label{sec:jacobi}

In this section, we compute the Jacobi equation 
of the planar $2$-sphere $\Sigma_4(a)\subset E(a,b,b,d)$, which corresponds to the 
geodesic $\gamma_\hor$ in $\Omega_a$, in the notation of the previous section, and write it as a Sturm--Liouville ODE. 
This is key to locate all degeneracies along the trivial branch \eqref{eq:trivial_branch}, since the linearizations of $f_\even$ and $f_\odd$ at $s=0$ are determined, respectively, by the endpoint values (corresponding to the intersection with $\gamma_\ver$) of a nontrivial solution to this ODE and its first derivative, see \eqref{eq:SL-boundary-0}.
The characterization of degeneracy instants along the trivial branch in terms of a Sturm--Liouville eigenvalue problem is given in Proposition~\ref{prop:char-degeneracies}.

\subsection{Jacobi fields}
Let $\gamma_\hor^+$ be the portion of the geodesic $\gamma_\hor$ joining $\beta(0)\in\partial\Omega_a$ to $O\in\gamma_\ver$, see Figure~\ref{fig:verhor}.
Consider the variation of $\gamma_\hor^+$ by the (reparametrized) geodesics
$[0,1)\ni z\mapsto \sigma(a,s,z)\in\Omega_a$, with $|s|<\varepsilon$, see \eqref{eq:sigmas}. Linearizing at $s=0$ and taking normal components, one obtains 
\begin{equation}\label{eq:JacobifieldJa}
J_a(z):=\tfrac{\partial}{\partial s}\sigma(a,s,z)\big|_{s=0}\in T_{\sigma(a,0,z)}\Omega_a,
\end{equation}
which is a Jacobi field along $\gamma_\hor^+$, with $x_4$-component in $T_{\sigma(a,0,z)}\Omega_a\subset \R^4$ given by
\begin{equation}\label{eq:v_a}
v_{a}\colon [0,1)\longrightarrow \R, \quad v_{a}(z):=\tfrac{\partial}{\partial s}\sigma_{x_4}(a,s,z)\big|_{s=0}.
\end{equation}

By Proposition~\ref{prop:reduction}, the preimage under $\Pi$ of this variation by geodesics gives rise to a variation by $\G$-invariant minimal surfaces in $(S^3,\g_a)$ of the planar hemisphere
\begin{equation}\label{eq:sigma_hor_+}
\Sigma_4^+(a) := \overline{\Pi^{-1}(\gamma_\hor^+)}= E(a,b,b,d) \cap \big\{0\leq x_1\leq a,\, x_4=0\big\}.
\end{equation}
The corresponding normal Jacobi field $\widetilde J_a$ along $\Sigma_4^+(a)$ projects via $\Pi$ to $J_a$. Thus, parametrizing 
$\Sigma_4^+(a)\setminus\{(a,0,0,0)\}$
with coordinates $(\rho,\theta) \in (0, T_a(0)] \times [0,2\pi]$ and denoting by $\vec n$ its unit normal in $(S^3,\g_a)$,
we have that the function $\psi_a\colon \Sigma_4^+(a)\to\R$, given by $\psi_a(\rho,\theta)=\g_a\big(\widetilde J_a,\vec n \big)$,
satisfies $\psi_a(t_a(z)\, T_a(0),\theta)=v_{a}(z)$ for all $z\in [0,1)$ and $\theta\in [0,2\pi]$, i.e., is a well-defined real analytic \emph{radial} function on  $\Sigma_4^+(a)$, and $\psi=\psi_a$
solves the Jacobi equation
\begin{equation}\label{eq:Jacobi_upstairs}
\Delta \psi - \Ric(\vec n)\,\psi=0,
\end{equation}
where $\Delta$ is the Laplacian on $\Sigma_4^+(a)$.
Since $\Sigma_4^+(a)$ 
is a portion of the totally geodesic $2$-sphere $\Sigma_4(a)$, 
the curvature term in \eqref{eq:Jacobi_upstairs} involves only the ambient Ricci curvature, because the second fundamental form vanishes identically, and $\Delta$ is the restriction to $\Sigma_4^+(a)$ of the (positive-semidefinite) Laplacian on $\Sigma_4(a)$.

\subsection{Sturm--Liouville equation}\label{sub:coeffSturmeqn}
In order to write the (radial part of the) Jacobi equation \eqref{eq:Jacobi_upstairs} as a Sturm--Liouville ODE, satisfied by $v_{a}(z)$, recall from \eqref{eq:sigma_hor_+} that $\Sigma_4^+(a)$ is the portion with $x_1\geq0$ of the $2$-dimensional ellipsoid of revolution
\begin{equation}\label{eq:sigma_hor}
\Sigma_4(a)= 
\left\{ (x_1,x_2,x_3,0)\in \R^4 :\frac{x_1^2}{a^2}+\frac{x_2^2}{b^2}+\frac{x_3^2}{b^2} =1 \right\}.
\end{equation}
The Ricci curvature of $(S^3,\g_a)$ at $(x_1,x_2,x_3,0)\in \Sigma_4(a)$ in the unit normal direction $\vec n=(0,0,0,1)$ can be computed with the Gauss equation and \eqref{eq:sigma_hor} to be
\begin{equation}\label{eq:ricci}
\Ric(\vec n) =
a^2\, \frac{a^2\left(1-\frac{x_1^2}{a^2}\right)+b^2\left(1+ \frac{x_1^2}{a^2}\right) }{d^2
\left(a^2\left(1-\frac{x_1^2}{a^2}\right) +b^2 \, \frac{x_1^2}{a^2} \right)^2 }.
\end{equation}

Moreover, the induced metric on $\Sigma_4(a)$, i.e., such that the inclusion \eqref{eq:sigma_hor}  into $(S^3,\g_a)$ is isometric, coincides with that induced by the parametrization
\begin{equation*}
x_1(\phi,\theta)=a\cos \phi, \quad 
x_2(\phi,\theta)=b\sin\phi \sin\theta, \quad
x_3(\phi,\theta)=b\sin\phi \cos\theta,
\end{equation*}
where $\phi\in [0,\pi]$ and $\theta\in [0,2\pi]$. Thus, it can be written in polar coordinates $(\rho,\theta)$ as $\dd \rho^2+\varphi(\rho)^2 \dd\theta^2$, where
\begin{equation*}
\dd \rho = \sqrt{a^2\sin^2\phi + b^2\cos^2\phi}\;\dd \phi, 
\quad \text{ and } \quad \varphi(\rho)=b\sin(\phi(\rho)). 
\end{equation*}
In particular, the Laplacian of a radial function $\psi\colon \Sigma_4(a)\to\R$ can be computed as
\begin{align}\label{eq:laplacian-rho-z}
\Delta \psi(\rho)  &= -\frac{1}{\varphi(\rho)}\frac{\dd}{\dd \rho}\left(\varphi(\rho)\frac{\dd}{\dd \rho} \psi(\rho)\right)\nonumber \\
&= -\frac{1}{\sin\phi \sqrt{a^2\sin^2\phi + b^2\cos^2\phi}} \frac{\dd}{\dd \phi}\left(\frac{\sin\phi}{\sqrt{a^2\sin^2\phi + b^2\cos^2\phi}} \frac{\dd}{\dd \phi} \psi(\rho(\phi))\!\right)\\
&= -\frac{1}{\sqrt{a^2(1-z^2)+b^2z^2}} \frac{\dd}{\dd z}\left(\frac{1-z^2}{\sqrt{a^2(1-z^2)+b^2z^2}} \frac{\dd}{\dd z}\psi\big(\rho(\phi(z))\big)\!\right),\nonumber
\end{align}
since $z=\frac{x_1}{a}=\cos\phi$, 
and so $\rho(\phi(z))=\int_0^{\arccos z} \sqrt{a^2\sin^2 \xi + b^2\cos^2\xi}\;\dd \xi$. Thus, we arrive at the desired characterization of radial solutions to the Jacobi equation:

\begin{proposition}\label{prop:SL-characterization}
A radial function $\psi\colon \Sigma_4^+(a)\to\R$, $\psi(\rho,\theta)=\psi(\rho)$, solves the Jacobi equation~\eqref{eq:Jacobi_upstairs} if and only if $v(z)=\psi(\rho(\phi(z)))$ solves the ODE
\begin{equation}\label{eq:SL}
-\frac{\dd}{\dd z}\left(p_a(z)\frac{\dd}{\dd z}v(z) \right)+q_a(z)\, v(z)=0, \quad z\in [0,1),
\end{equation}
where
\begin{equation}\label{eq:pa-qa}
p_a(z) :=\dfrac{ 1-z^2 }{\sqrt{a^2(1-z^2)+b^2z^2}}, \quad
q_a(z) :=-a^2\dfrac{a^2(1-z^2) +b^2(1+z^2)}{d^2 \big(a^2(1-z^2)+b^2z^2\big)^{3/2}}.
\end{equation}
In particular, this is the case for $\psi_a=\g_a\big(\widetilde J_a,\vec n \big)$, and 
$v_{a}(z):=\psi_{a}(\rho(\phi(z)))$. 
\end{proposition}

Clearly, both $p_a(z)$ and $q_a(z)$ are real analytic functions on $[0,1]$, and admit an even real analytic extension to $[-1,1]$, given by the same expressions above, and denoted by the same symbols. Straightforward computations show that
\begin{equation}\label{eq:minmaxpaqa}
\begin{aligned}
0&\leq p_a(z)\leq \tfrac{1}{a}, \\
\min\left\{-\tfrac{a^2+b^2}{a d^2},-\tfrac{2a^2}{b d^2} \right\} &\leq q_a(z)\leq \max\left\{-\tfrac{a^2+b^2}{a d^2},-\tfrac{2a^2}{b d^2} \right\},
\end{aligned}
\quad \text{ for all } 0\leq z\leq 1,
\end{equation}
and $p_a(z)$ has a simple zero at $z=1$, so \eqref{eq:SL} has a \emph{regular singular endpoint}; this is discussed further in Section~\ref{sec:singularSL}.
Moreover, from \eqref{eq:pa-qa}, the expansions of $p_a(z)$ and $q_a(z)$ as power series centered at $z=1$,
\begin{equation*}
p_a(z)=\sum\limits_{n=1}^\infty \hat p_{a}^{(n)}(z-1)^n, \quad\text{ and }\quad q_a(z)=\sum\limits_{n=0}^\infty \hat q_{a}^{(n)}(z-1)^n,
\end{equation*}
have radius of convergence equal to
\begin{equation}\label{eq:Ra}
R_a:=\begin{cases}\dfrac{\max\{a,b\}}{\sqrt{\vert a^2-b^2\vert}}-1,& \text{if } a\ne b,\\[.3cm]
+\infty,& \text{if } a=b.
\end{cases}
\end{equation}
It is easy to see that, for all $r\in(0,R_a)$ and  $n\in\mathds N$, the maps $(0,+\infty)\ni a\mapsto r^n\vert \hat p_{a}^{(n)}\vert$ and
$(0,+\infty)\ni a\mapsto r^n\vert \hat q_{a}^{(n)}\vert$ are real analytic, and locally bounded uniformly in $n$.

Finally, let us observe that $p_a$ and $q_a$ depend monotonically on $a$, since
\begin{equation}\label{eq:dpaqanegative}
\frac{\partial}{\partial a}p_a(z) < 0\quad\text{and}\quad \frac{\partial}{\partial a}q_a(z)<0,\quad\text{for all}\ z\in (0,1).
\end{equation}
By \eqref{eq:minmaxpaqa}, these functions converge uniformly on $z\in [0,1]$ as $a\nearrow+\infty$ as follows:
\begin{equation}\label{eq:paqalimit}
\lim_{a\to+\infty}p_a(z)=0,\quad\text{and}\quad\lim_{a\to+\infty}q_a(z)=-\infty.
\end{equation}

\subsection{Boundary conditions}
Let us now consider the behavior of the solution $v_a(z)$ to the Sturm--Liouville ODE \eqref{eq:SL} at the endpoints $z=0$ and $z=1$, which correspond to where $\gamma_\hor^+$ meets $O$ and $\partial\Omega_a$, respectively.

First, by Definition~\ref{def:fevenfodd}, the boundary conditions satisfied by $v_a(z)$ at $z=0$ are:
\begin{equation}\label{eq:SL-boundary-0}
\begin{aligned}
v_a(0)&=\tfrac{\partial}{\partial s}\sigma_{x_4}(a,s,0)\big|_{s=0}=\tfrac{\partial}{\partial s}  f_\odd(a,s)\big|_{s=0},\\[4pt]
v'_a(0)&=\tfrac{\partial}{\partial z}\tfrac{\partial}{\partial s}\sigma_{x_4}(a,s,z)\big|_{z=0,s=0}=\tfrac{\partial}{\partial s}  f_\even(a,s)\big|_{s=0}.
\end{aligned}
\end{equation}

Second, let us analyze the (limiting) boundary conditions satisfied by $v_{a}(z)$ at $z=1$.
Since $\lim_{z\nearrow1}\sigma(a,s,z)=\beta(s)$ for all $s\in\mathcal I$, recall \eqref{eq:beta}, the linearization of $\sigma_{x_4}$ at $s=0$
satisfies $\lim_{z\nearrow1} v_{a}(z)= d$.
Moreover, for all $s\in \mathcal I$, the (reparametrized) geodesic $\sigma(a,s,z)$ meets $\partial\Omega_a$ orthogonally at $\beta(s)$ as $z\nearrow1$. So, 
\begin{align*}
0&= \check{\g}_a\!\left(\lim_{z\nearrow1} \tfrac{\partial}{\partial z}\sigma(a,s,z) , \,\beta'(s)\!\right)\\
&= -a \sin s \,\lim_{z\nearrow1} \tfrac{\partial}{\partial z}\sigma_{x_1}(a,s,z)  + d \cos s \, \lim_{z\nearrow1}\tfrac{\partial}{\partial z}\sigma_{x_4}(a,s,z).
\end{align*}
Linearizing the above at $s=0$, since $\sigma_{x_1}(a,0,z)=az$ for all $z\in[0,1)$, we see that
\begin{equation*}
\tfrac{\partial}{\partial s}\Big(\lim_{z\nearrow1} \tfrac{\partial}{\partial z}\sigma_{x_4}(a,s,z) \Big)\Big|_{s=0}=\tfrac{a}{d}\lim_{z\nearrow1} \tfrac{\partial}{\partial z} \sigma_{x_1}(a,0,z)=\tfrac{a^2}{d},
\end{equation*}
hence $\lim_{z\nearrow1} v_{a}'(z)=\frac{a^2}{d}$.
Thus, altogether, the boundary conditions at $z=1$ are:
\begin{equation}\label{eq:SL-boundary}
\lim_{z\nearrow1} \; v_{a}(z)=d, \qquad \text{ and }\qquad \lim_{z\nearrow1} \; v'_{a}(z)=\tfrac{a^2}{d}.
\end{equation}

\begin{remark}
A radial function $\psi(\rho,\theta)=\psi(\rho)$ as in Proposition~\ref{prop:SL-characterization} is of class $C^1$ at $\rho=0$ if and only if $\psi'(0)=0$. This \emph{intrinsic} smoothness condition automatically holds whenever $\psi(\rho)$ is defined by means of a $C^1$ function $v(z)$ 
as $\psi(\rho(\phi(z)))=v(z)$, independently of \eqref{eq:SL-boundary}. Indeed, differentiating both sides in $z$ yields:
\begin{equation*}
\psi'\big(\rho(\phi(z))\big) = - \frac{ \sqrt{1-z^2}}{ \sqrt{a^2 (1-z^2) + b^2 z^2 } }\, v'(z),
\end{equation*}
which converges to zero as $z\nearrow 1$, i.e., as $\rho\searrow0$.
Similar considerations can be made
regarding \eqref{eq:SL-boundary-0} and the above as $z\searrow0$, 
related to the existence of even/odd $C^1$ extensions of $\psi \colon \Sigma_4^+(a)\to\R$ to all of $\Sigma_4(a)$.
\end{remark}

\section{Singular Sturm--Liouville eigenvalue problems}\label{sec:singularSL}

Motivated by the Sturm--Liouville equation \eqref{eq:SL} in Proposition~\ref{prop:SL-characterization}, consider the differential operator  $\mathcal L_a$ defined on smooth functions $v\colon [0,1)\to\R$ by
\begin{equation}\label{eq:La}
\mathcal L_a(v)=\frac{1}{p_a}\big(-(p_a\,v')'+q_a\,v\big),
\end{equation}
where $p_a$ and $q_a$ are the functions defined in \eqref{eq:pa-qa}.
An application of the classical \emph{Frobenius method} (see, e.g., \cite[Sec.~4.2]{teschl})
yields the following existence result: 

\begin{proposition}\label{prop:r-asolution}
For all $a>0$ and $\lambda\in\R$, there exists a unique real analytic function $u_{a,\lambda}\colon [0,1]\to\R$ such that $\mathcal L_a(u_{a,\lambda})=\lambda\, u_{a,\lambda}$,
$u_{a,\lambda}(1)>0$ and 
\begin{equation}\label{eq:ualambdanorm}
\int_0^1 u_{a,\lambda}(z)^2\,p_a(z)\dd z=1.
\end{equation}
Furthermore, $a^2 u_{a,\lambda}(1)-d^2 u'_{a,\lambda}(1)=0$, and the map $(a,\lambda,z)\mapsto u_{a,\lambda}(z)$ is real analytic.
Any solution to $\mathcal L_a(v)=\lambda \, v$ which is not a multiple of $u_{a,\lambda}(z)$ is of the form $v(z)=C\,u_{a,\lambda}(z)\log(1-z)+w(z)$, with $C\neq0$ and $w\colon[0,1]\to\R$ real analytic.
\end{proposition}

\begin{proof}
The indicial equation for the singularity $z=1$ of the ODE 
$\mathcal L_a(v)=\lambda\,v$ has $r=0$ as a double root, which implies the existence of a unique (up to constant factors) power series solution centered at $z=1$, whose radius of convergence is equal to $R_a$, see \eqref{eq:Ra}. Since $p_a$ and $q_a$ are real analytic on $[0,1]$, any power series solution of this ODE centered at $z=1$ can be uniquely extended to a real analytic solution on all of $[0,1]$. Thus, $u_{a,\lambda}\colon [0,1]\to\R$ is uniquely defined by the above normalizations. The claim characterizing unbounded solutions is standard~\cite{teschl}.

Evaluating both sides of $\mathcal L_a(v)=\lambda\,v$ at $z=1$ shows that 
a necessary condition for $v\colon [0,1]\to\R$ to be a real analytic solution is that
$q_a(1) v(1)-p_a'(1)v'(1)=0$, since $p_a(1)=0$.
Thus, as $q_a(1)=-\frac{2a^2}{bd^2}$ and $p'_a(1)=-\frac{2}{b}$, 
any such solution; in particular $v=u_{a,\lambda}$, satisfies $a^2v(1)-d^2v'(1)=0$.
Regarding analytic dependence, after the value $u_{a,\lambda}(1)$ is chosen, all the coefficients of the power series $u_{a,\lambda}(z)=\sum_{n\ge0}\hat u_{a,\lambda}^{(n)}(z-1)^n$ can be determined inductively from those of $p_a$ and $q_a$ using this initial condition, and they are easily seen to be real analytic functions of $a$ and $\lambda$. 
Since the same holds for the coefficients of $p_a$ and $q_a$, it follows that, for all $r\in(0,R_a)$ and $n\in\mathds N$, the maps $(a,\lambda)\mapsto r^n\big\vert \hat u_{a,\lambda}^{(n)}\big\vert$ are locally bounded uniformly in $n$, which implies that a solution with fixed value at $z=1$ depends in a real analytic way on the pair $(a,\lambda)$. Clearly, the same holds under the normalization \eqref{eq:ualambdanorm}.
\end{proof}

\begin{remark}
If $a=b=c=d$, the function $u_{a,0}(z)$ is explicitly computed in \eqref{eq:ua0round}.
\end{remark}

\subsection{Eigenvalue problems and spectra}\label{subsec:evalprobs}
As unbounded solutions to $\mathcal L_a(v)=0$ have a logarithmic singularity at $z=1$ by Proposition~\ref{prop:r-asolution}, it follows that the regular singular endpoint $z=1$ is of \emph{limit circle non-oscillating} type, see e.g.~\cite{weidmann,zettl}.
In particular, by Weyl--Titchmarsh--Kodaira spectral theory, the operator $\mathcal L_a$ is (essentially) self-adjoint in the following spaces (see~\cite[\S 4]{weidmann} or \cite[\S 6]{egnt}):
\begin{equation}\label{eq:vevenvodd}
\begin{aligned}
V_\even &:= \Big\{ v\in C^\infty([0,1),\R) : v'(0)=0, \;  \lim_{z\nearrow1 } p_a(z)v'(z)=0 \Big\}, \\
V_\odd &:= \Big\{ v\in C^\infty([0,1),\R) : \, v(0)=0, \; \lim_{z\nearrow1 } p_a(z)v'(z)=0\Big\}.
\end{aligned}
\end{equation}

We denote by $(\mathcal P_a)_\even$ and $(\mathcal P_a)_\odd$
the Sturm--Liouville eigenvalue problem $\mathcal L_a(v)=\lambda \, v$ on the spaces $V_\even$ and $V_\odd$, respectively.
Despite being \emph{singular}, these problems enjoy virtually all the usual spectral properties of \emph{regular} Sturm--Liouville problems.
For the reader's convenience, we discuss some of these results:

\begin{lemma}\label{lem:bounded-eigenfunctions}
If $v\in V_\even$ or $v\in V_\odd$ is an eigenfunction of $\mathcal L_a$, i.e., $\mathcal L_a(v)=\lambda\,v$, $\lambda\in\R$, then $v(z)$ admits a smooth extension to $z=1$ and $a^2v(1)-d^2v'(1)=0$.
\end{lemma}

\begin{proof}
By Proposition~\ref{prop:r-asolution}, either $v(z)$ is real analytic at $z=1$, so $v(z)=C\, u_{a,\lambda}(z)$, $C\in\R$, and  the stated conclusions hold, or else $v(z)=C\,u_{a,\lambda}(z) \log(1-z)+w(z)$, where $w\colon[0,1]\to\R$ is real analytic and $C\neq 0$. 
The latter functions $v(z)$ do not satisfy $\lim_{z\nearrow1}p_a(z)v'(z)=0$, and hence do not belong to $V_\even$ nor $V_\odd$.
\end{proof}

\begin{remark}
Furthermore, note that eigenfunctions $v\colon[0,1)\to\R$ of $(\mathcal P_a)_\even$ and $(\mathcal P_a)_\odd$ also extend across $z=0$ to functions $v\colon [-1,1]\to\R$ satisfying $\mathcal L_a(v)=\lambda\,v$ that are even and odd, respectively, since $p_a$ and $q_a$ are both even functions.
\end{remark}

The resolvents of $\mathcal L_a$ on both $V_\even$ and $V_\odd$ are compact (Hilbert--Schmidt) operators~\cite[Thm.~7.1]{egnt}. Thus, applying the Spectral Theorem and recalling standard oscillation results for singular Sturm--Liouville operators, we arrive at the following statement, see also~\cite[Thm.~10.12.1, (3), (4), p.\ 208--209]{zettl}.

\begin{proposition}\label{thm:spectralthsingular}
For all $a>0$, the spectra of $(\mathcal P_a)_\even$ and $(\mathcal P_a)_\odd$ 
are discrete, bounded from below, unbounded from above, and every eigenvalue is simple:
\begin{equation*}
\begin{aligned}
&\text{Spectrum of } (\mathcal P_a)_\even: & \lambda_0^\even(a)<\lambda_1^\even(a)<\dots<\lambda_n^\even(a)<\dots\nearrow+\infty, \\
&\text{Spectrum of } (\mathcal P_a)_\odd: & \lambda_0^\odd(a)<\lambda_1^\odd(a)<\dots<\lambda_n^\odd(a)<\dots\nearrow+\infty.
\end{aligned}
\end{equation*}
The eigenfunctions of $(\mathcal P_a)_\even$ as well as those of $(\mathcal P_a)_\odd$ form orthogonal bases of the Hilbert space $L^2([0,1],p_a\dd z)$. 
In particular, the number of negative eigenvalues of $(\mathcal P_a)_\even$ and $(\mathcal P_a)_\odd$ is equal to the dimension of the largest subspace in $V_\even$ and $V_\odd$, respectively, on which the following quadratic form is negative-definite:
\begin{equation}\label{eq:defQa}
Q_a(v):=\langle \mathcal L_a(v),v\rangle 
= \int_0^1 p_a(z)v'(z)^2+q_a(z) v(z)^2  \;\dd z.
\end{equation}
Moreover, $n$-th eigenfunctions of $(\mathcal P_a)_\even$ and $(\mathcal P_a)_\odd$ have exactly $n$ zeros in~$(0,1)$.
\end{proposition} 

Clearly, the right-hand side of \eqref{eq:defQa} also defines a continuous quadratic form on the Sobolev space $W^{1,2}([0,1],\R)$. 
In order to estimate the index of $Q_a$ in the spaces \eqref{eq:vevenvodd}, we shall use test functions in $W^{1,2}([0,1],\R)$ and standard density arguments.

\subsection{Degeneracy instants}
Analyzing the spectra of $(\mathcal P_a)_\even$ and $(\mathcal P_a)_\odd$, we may determine whether the linearization of $f_\even$ and $f_\odd$ at $(a,0)$ vanishes.
Namely, by Proposition~\ref{prop:SL-characterization}, Lemma~\ref{lem:bounded-eigenfunctions}, and \eqref{eq:SL-boundary-0} and \eqref{eq:SL-boundary}, we have the following:

\begin{proposition}\label{prop:char-degeneracies}
The instant $a=a_*$ is a degeneracy instant for $f(a,s)=0$, i.e., $\frac{\partial f}{\partial s}(a_*,0)=0$, where $f=f_\even$ or $f_\odd$, if and only if $\lambda=0$ is an eigenvalue of $(\mathcal P_{a_*})_\even$ or $(\mathcal P_{a_*})_\odd$, respectively. In this case, $\ker \mathcal L_{a_*}$ is spanned by $u_{a_*,0}(z)$, which is a constant multiple of the function $v_{a_*}(z)$ defined in \eqref{eq:v_a}. 
\end{proposition}

Note also that the number of negative eigenvalues of $(\mathcal P_a)_\even$ or $(\mathcal P_a)_\odd$ is the $\mathsf{SO}(2)\times\Z_2$-equivariant Morse index of $\Sigma_4(a)\subset E(a,b,b,d)$ as a minimal surface, where $\mathsf{SO}(2)$ acts with orbits \eqref{g-orbits}, and $\Z_2\cong\{\pm 1\}$ acts as $-1\cdot (x_1,x_2,x_3,x_4)\mapsto (-x_1,x_2,x_3,\pm x_4)$ where $\pm$ is $+$ in the even case, and $-$ in the odd case.

\subsection{Spectral analysis}
The spectra of $(\mathcal P_{a})_\even$ and $(\mathcal P_{a})_\odd$ depend not only on  $a$, but also on $b=c$ and $d$, which are omitted to simplify notation. Nevertheless, 
$\lambda_n^\even(a)$ and $\lambda_n^\odd(a)$ satisfy important monotonicity properties on both $a$ and $d$:

\begin{proposition}\label{prop:negdereigenv}
For all $n\geq0$, the $n$-th eigenvalues $\lambda_n^\even$ and $\lambda_n^\odd$ are strictly decreasing real analytic functions  of $a\in(0,+\infty)$, and strictly increasing real analytic functions of $d\in(0,+\infty)$.
\end{proposition}

\begin{proof}
Let $\overline\lambda\in\R$ be an eigenvalue of either $(\mathcal P_{\overline a})_\even$ or $(\mathcal P_{\overline a})_\odd$. First, we focus on the dependence on the parameter $a$.
We shall apply the Implicit Function Theorem to the equation $u_{a,\lambda}'(0)=0$ near $(\overline a,\overline\lambda)$ in the even case, and $u_{a,\lambda}(0)=0$ in the odd case, where $u_{a,\lambda}\colon [0,1]\to\R$ is the real analytic function defined in Proposition~\ref{prop:r-asolution}.
Multiplying both sides of $\mathcal L_a(u_{a,\lambda})=\lambda\,u_{a,\lambda}$ by $p_a\,u_{a,\lambda}$ and integrating by parts on $[0,1]$, together with \eqref{eq:ualambdanorm} and \eqref{eq:defQa}, yields: 
\begin{equation}\label{eq:lambda=int}
\lambda= p_a(0)u_{a,\lambda}'(0)u_{a,\lambda}(0)+ Q_a(u_{a,\lambda}).
\end{equation}
Differentiating the above with respect to $\lambda$ gives
\begin{multline}\label{eq:1=int}
1=p_a(0)\,\tfrac{\partial u_{a,\lambda}'}{\partial\lambda}(0)\,u_{a,\lambda}(0)+p_a(0)u_{a,\lambda}'(0)\,\tfrac{\partial u_{a,\lambda}}{\partial\lambda}(0) \\
+2\int_0^1  p_a\, u_{a,\lambda}'\left(\tfrac{\partial u_{a,\lambda}}{\partial\lambda}\right)'
+q_a\,u_{a,\lambda}\,\tfrac{\partial u_{a,\lambda}}{\partial\lambda} \; \dd z;
\end{multline}
while, differentiating \eqref{eq:ualambdanorm} with respect to $\lambda$, we have 
$\int_0^1 u_{a,\lambda}\,\tfrac{\partial u_{a,\lambda}}{\partial\lambda}\;p_a\dd z=0$.
Using the above and integration by parts in \eqref{eq:1=int}, recalling that $p_a(1)=0$, we have:
\begin{equation}\label{eq:lastpartialintegr}
\textstyle 1=-p_{a}(0)\,u_{a,\lambda}'(0)\,\frac{\partial u_{a,\lambda}}{\partial\lambda}(0)+p_{a}(0)\,\frac{\partial u_{a,\lambda}'}{\partial\lambda}(0)\,u_{a,\lambda}(0).
\end{equation}
Thus, if $u_{a,\lambda}(0)=0$, then $\frac{\partial u_{a,\lambda}}{\partial\lambda}(0)\ne0$; while if $u'_{a,\lambda}(0)=0$, then $\frac{\partial u_{a,\lambda}'}{\partial\lambda}(0)\ne0$. Therefore, the Implicit Function Theorem applies in both even and odd cases, giving $\varepsilon>0$ and a real analytic function $\lambda\colon(\overline a-\varepsilon,\overline a+\varepsilon)\to\R$ such that $\lambda(a)$ is an eigenvalue of
the corresponding problem $(\mathcal P_{a})$ for all $a\in (\overline a-\varepsilon,\overline a+\varepsilon)$, and $\lambda(\overline a)=\overline\lambda$. 
The derivative of $\lambda(a)$ is computed analogously, by differentiating \eqref{eq:lambda=int} with respect to $a$. Namely, taking \eqref{eq:ualambdanorm} into account, integration by parts yields:
\begin{equation}\label{eq:derlambdaa}
\frac{\partial\lambda}{\partial a}=\int_0^1 \frac{\partial p_a}{\partial a}\big(u_{a,\lambda(a)}'\big)^2+\frac{\partial q_a}{\partial a}\,u_{a,\lambda(a)}^2 \;\mathrm dz <0,
\end{equation}
both in the even and odd cases, and the last inequality follows from \eqref{eq:dpaqanegative}.
Furthermore, each of the functions $\lambda$, i.e., $\lambda^\even_n$ and $\lambda^\odd_n$,
are locally bounded near any $a\in(0,+\infty)$ by Proposition~\ref{thm:spectralthsingular}, so, by the above, they are globally defined strictly decreasing real analytic functions of $a\in(0,+\infty)$.

Similarly, the above arguments can be easily adapted to show that $\lambda^\even_n$ and $\lambda^\odd_n$ are globally defined real analytic functions of the parameter $d\in (0,+\infty)$. 
From \eqref{eq:lambda=int} and \eqref{eq:pa-qa}, we compute their derivative in $d$ analogously to \eqref{eq:derlambdaa}, obtaining:
\begin{equation*}
\frac{\partial\lambda}{\partial d}=\int_0^1\frac{\partial q_a}{\partial d}\,u_{a,\lambda}^2 \;\dd z=-\frac2d\int_0^1q_a\,u_{a,\lambda}^2 \;\dd z >0.\qedhere
\end{equation*}
\end{proof}

\begin{proposition}\label{prop:degsequences}
The set of instants $a>0$ such that $\lambda=0$ is an eigenvalue of problem $(\mathcal P_a)_\even$, respectively $(\mathcal P_a)_\odd$, cf.~Proposition~\ref{prop:char-degeneracies}, is an unbounded strictly increasing sequence, that we denote $(a_n^\even)_{n\geq1}$, respectively $(a_n^\odd)_{n\geq0}$.
\end{proposition}

\begin{proof}
By Proposition~\ref{prop:r-asolution}, the sets in question are the zero sets of the real analytic functions $(0,+\infty)\ni a\mapsto u_{a,0}'(0)$, respectively $(0,+\infty)\ni a\mapsto u_{a,0}(0)$. These are closed discrete subsets of $(0,+\infty)$, since the above functions are nonconstant by Proposition~\ref{prop:negdereigenv}.
Furthermore, we claim that for all $\overline a>0$, each of these sets contains only finitely many points in the interval $(0,\overline a)$. If not, the corresponding problem $(\mathcal P_{\overline a})$ would have infinitely many negative eigenvalues, by the monotonicity property in $a$ from Proposition~\ref{prop:negdereigenv}. However, this contradicts the fact that its spectrum is closed, discrete and bounded from below (Proposition~\ref{thm:spectralthsingular}).

It only remains to show that the sets in question are unbounded.
Using again the monotonicity of the eigenvalues in $a$ from Proposition~\ref{prop:negdereigenv}, it suffices to show that the corresponding problem $(\mathcal P_a)$ has arbitrarily many negative eigenvalues
as $a\nearrow+\infty$. This follows by exhibiting subspaces of $V_\even$ and $V_\odd$ on which the quadratic form $Q_a$ defined in \eqref{eq:defQa} is negative-definite, whose dimension can be made arbitrarily large if $a$ is sufficiently large. (This was also observed by Haslhofer--Ketover~\cite[Prop.~9.3]{hk}.)
Namely, for $\alpha,\delta,\varepsilon>0$ with $\alpha+2\delta+\varepsilon<1$,\linebreak let $\xi_{\alpha,\delta,\varepsilon}\colon[0,1]\to[0,1]$ denote the piecewise affine function that is equal to $1$ on $[\alpha+\delta,\alpha+\delta+\varepsilon]$ and vanishes on $[0,\alpha]\cup[\alpha+2\delta+\varepsilon,1]$.
By adjusting $\alpha$, $\delta$ and $\varepsilon$ judiciously, one can construct an arbitrarily large number of $\xi_{\alpha,\delta,\varepsilon}$ with pairwise disjoint supports. 
Using either \eqref{eq:minmaxpaqa} or \eqref{eq:paqalimit}, it is easy to see that
$Q_a(\xi_{\alpha,\delta,\varepsilon})\leq -1< 0$, for all $\alpha,\delta,\varepsilon$ as above and $a$ sufficiently large, thus $Q_a$ is negative-definite on the subspace of $W^{1,2}([0,1],\R)$ spanned by these functions. 
Thus, by standard density arguments, $Q_a$ is also negative-definite on subspaces of $V_\even \cap V_\odd$ with arbitrarily large dimension, provided that $a$ is sufficiently large.
\end{proof}

In the sequel, we will also need the following result similar to Proposition~\ref{prop:negdereigenv}:

\begin{lemma}\label{lemma:deranonzero}
For all $\lambda\in\mathds R$, the zeros of $a\mapsto u_{a,\lambda}(0)$ and $a\mapsto u'_{a,\lambda}(0)$ are simple. 
\end{lemma}

\begin{proof}
Arguing as in the proof of Proposition~\ref{prop:negdereigenv}, differentiating \eqref{eq:lambda=int}  with respect to $a$, and integrating by parts yields:
\begin{multline*}
p_a(0)u_{a,\lambda}(0) \tfrac\partial{\partial a}u'_{a,\lambda}(0) -p_a(0)u'_{a,\lambda}(0) \tfrac\partial{\partial a}u_{a,\lambda}(0)\\=-\int_0^1 \tfrac{\partial p_a}{\partial a}(u_{a,\lambda}')^2+\tfrac{\partial q_a}{\partial a}(u_{a,\lambda})^2\;\dd z >0,
\end{multline*}
where the inequality follows from \eqref{eq:dpaqanegative}.
The conclusion follows, as $p_a(0)=\frac{1}{a}$.
\end{proof}

Let us analyze two special situations in which further information on the spectra of $(\mathcal P_a)_\even$ and $(\mathcal P_a)_\odd$ may be inferred from the presence of additional symmetries. 
The first   is $a=b=c=d$, in which case $(S^3,\g_a)$ is a round $3$-sphere of radius $a$, centered at the origin in $\R^4$. In particular, $\Sigma_4(a)$ is an equator and isometric to a round $2$-sphere of radius $a$, so its Laplacian $\Delta$ has eigenvalues 
$\tfrac{1}{a^2} k(k+1)$, for $k\in\N_0$.
The Ricci curvature of $(S^3,\g_a)$ is constant and equal to $\frac{2}{a^2}$ in all unit directions, cf.~\eqref{eq:ricci}, so the Jacobi operator on $\Sigma_4(a)$ is $\Delta-\frac{2}{a^2}$, and its eigenvalues~are:
\begin{equation*}
\tfrac{1}{a^2} (k(k+1)-2), \text{ for all } k\in\N_0.
\end{equation*}
Setting $k=0$, we have the simple eigenvalue $-\frac{2}{a^2}$, whose eigenspace is spanned by constant functions, induced by translations in the $x_4$-direction. Setting $k=1$, we have the null eigenvalue, of multiplicity $3$, whose eigenspace is spanned by Jacobi fields induced by the Killing fields of $(S^3,\g_a)$ given by rotations preserving~$x_4=0$.

By the arguments leading to Proposition~\ref{prop:SL-characterization}, a \emph{radial} function 
$\psi\colon\Sigma_4^+(a)\to\R$, $\psi(\rho,\theta)=\psi(\rho)$, is an eigenfunction of the (restriction to $\Sigma_4^+(a)$ of the) Jacobi operator $\Delta-\frac{2}{a^2}$ with eigenvalue $\mu$ if and only if $v(z)=\psi(\rho(\phi(z)))$ solves
\begin{equation}\label{eq:Jacobi-ODE-geometric}
\textstyle -\frac{\dd}{\dd z}\!\left( \frac{1-z^2}{a}\frac{\dd}{\dd z} v(z) \right)- \frac2a\, v(z) =\mu \,a\, v(z), \quad z\in [0,1),
\end{equation}
where $\rho(\phi(z))=a\,\arccos z$, i.e., $z=\cos\frac{\rho}{a}$.
The left-hand side of the above ODE is precisely $p_a\mathcal L_a(v)$, since the functions in \eqref{eq:pa-qa} simplify to $p_a(z)=\frac{1-z^2}{a}$ and $q_a(z)=-\frac{2}{a}$, if one assumes $a=b=c=d$. Studying the above ODE leads us to:

\begin{proposition}\label{prop:firsteigenvalues}
If $a=b=c=d$, then $\lambda_0^\even(a)<0$, and $\lambda_0^\odd(a)=0$.
\end{proposition}

\begin{proof}
The function $v(z)=a z$ belongs to $V_\odd$ and solves the ODE \eqref{eq:Jacobi-ODE-geometric} with $\mu=0$, which coincides with the ODE in $(\mathcal P_a)_\odd$ with $\lambda=0$.  
Since it has no zeros in $(0,1)$, by Proposition~\ref{thm:spectralthsingular}, this is the $0$-th odd eigenfunction, i.e., $\lambda_0^\odd(a)=0$.

The constant function $v(z)\equiv 1$ belongs to $V_\even$ and solves the ODE \eqref{eq:Jacobi-ODE-geometric} with $\mu=-\frac{2}{a^2}$. Although this ODE, which can be rewritten as $p_a\mathcal L_a(v)=-\frac{2}{a} v$, \emph{does~not} coincide with $\mathcal L_a(v)=\lambda\,v$, both ODEs have the same number of negative eigenvalues. More concretely, by \eqref{eq:minmaxpaqa}, the quadratic form $Q_a$ is negative-definite on the subspace spanned by $v(z)\equiv 1$, thus $\lambda_0^\even(a)<0$ by Proposition~\ref{thm:spectralthsingular}.
\end{proof}

\begin{corollary}\label{cor:even<odd}
The inequality $\lambda_0^\even(a)<\lambda_0^\odd(a)$ holds for all $a>0$.
\end{corollary}

\begin{proof}
The stated inequality holds if $a=b=c=d$ by Proposition~\ref{prop:firsteigenvalues}. 
Since the eigenvalues $\lambda_0^\even$ and $\lambda_0^\odd$ depend continuously on all parameters $a,b,c,d$, cf.~Proposition~\ref{prop:negdereigenv}, if this inequality were to fail for some $a,b,c,d$, then there would exist values of those parameters for which $\lambda_0^\even(a)=\lambda_0^\odd(a)=\lambda$ is simultaneously an eigenvalue for $(\mathcal P_a)_\even$ and $(\mathcal P_a)_\odd$. This is  impossible, because the only solution to $\mathcal L_a(v)=\lambda \, v$ with $v(0)=v'(0)=0$ is clearly $v\equiv 0$.
\end{proof}

\begin{remark}
From the proof of Proposition~\ref{prop:firsteigenvalues} and \eqref{eq:SL-boundary}, if $a=b=c=d$, then \eqref{eq:v_a} is given by $v_a(z)=a z$.
Since $\int_0^1 v_a(z)^2\;p_a(z)\dd z=\frac{2a}{15}$, the function $u_{a,0}(z)$ in Proposition~\ref{prop:r-asolution}, which is a constant multiple of $v_a(z)$ satisfying \eqref{eq:ualambdanorm}, is given by 
\begin{equation}\label{eq:ua0round}
u_{a,0}(z)=\sqrt{\tfrac{15a}{2}} \,z.
\end{equation}
Although all rotations preserving the hyperplane $x_4=0$ induce Jacobi fields on $\Sigma_4(a)$, 
only those in the $(x_2,x_3)$-plane induce Jacobi fields that are \emph{radial} functions on $\Sigma_4(a)$. This explains the drop in 
multiplicity from $3$ to $1$ when comparing the spectrum of the Jacobi operator $\Delta-\frac{2}{a^2}$ on $\Sigma_4(a)$ and the spectrum of $(\mathcal P_a)_\odd$.
\end{remark}

\begin{remark}\label{rem:unbounded-solutions}
It is easy to find \emph{all solutions} to \eqref{eq:Jacobi-ODE-geometric} for the above values of $\mu$:
\begin{equation*}
\begin{aligned}
\mu=0: \qquad u(z)&=C_1\, z + C_2\,\big(1-z \operatorname{arctanh}(z)\big), \\
\mu=-\tfrac{2}{a^2}: \qquad u(z)&=C_1\, + C_2\,\operatorname{arctanh}(z), 
\end{aligned}
\end{equation*}
where $C_1,C_2\in\R$, 
and $\operatorname{arctanh}(z)=\log\sqrt\frac{1+z}{1-z}$, cf.~Proposition~\ref{prop:r-asolution}.
The above are real analytic on $[0,1]$ if and only if they are bounded; i.e., if and only if $C_2=0$.
\end{remark}

The second special case we analyze is when
$a=d$ and $b=c$, but these two values need not coincide. In this situation, the same conclusion of Proposition~\ref{prop:firsteigenvalues} holds:

\begin{proposition}\label{prop:firsteigenvalues2}
If $a=d$ and $b=c$, then $\lambda_0^\even(a)<0$, and $\lambda_0^\odd(a)=0$.
\end{proposition}

\begin{proof}
If $a=d$, rotations in the $(x_1,x_4)$-plane induce isometries of $\Omega_a$ that fix $O$, and 
map $\gamma_\hor=\gamma_{a,0}$ to $\gamma_{a,s}$ for any $s\in\R$, cf.~Theorem~\ref{thm:existence_gammas}. The associated Killing field induces a Jacobi field along $\gamma_\hor$, which, by Proposition~\ref{prop:SL-characterization} and \eqref{eq:SL-boundary}, determines a real analytic solution $v_a\colon [0,1]\to\R$ to $\mathcal L_a(v)=0$ with $v_a(0)=0$, $v_a'(1)=v_a(1)=a$, and $v_a(z)\neq 0$ for all $0<z\leq 1$, cf.~also Proposition~\ref{prop:char-degeneracies}.
Thus, by Proposition~\ref{thm:spectralthsingular}, this is the $0$-th eigenfunction of $(\mathcal P_a)_\odd$, i.e., $\lambda_0^\odd(a)=0$. (Note that the same argument proves the corresponding claim in Proposition~\ref{prop:firsteigenvalues}, but the extra assumption $a=b$ made there allows to easily find the eigenfunction $v_a(z)$ explicitly.) The claim that $\lambda_0^\even(a)<0$ now follows from Corollary~\ref{cor:even<odd}.
\end{proof}

\begin{corollary}\label{cor:firstdegen}
For any fixed $b=c>0$ and $d>0$, the first degeneracy instant for $f_\odd(a,0)=0$ is $a_0^\odd=d$, cf.~Propositions~\ref{prop:char-degeneracies} and \ref{prop:degsequences}. This instant is trivially a bifurcation instant, since $f_\odd(d,s)=0$ for all $s\in \mathcal I$.
\end{corollary}

\begin{proof}
If $0<a<d$, then $\lambda_0^\odd(a)>0$ by Propositions~\ref{prop:negdereigenv} and \ref{prop:firsteigenvalues2}. Therefore, since $\lambda_0^\odd(d)=0$ by Proposition~\ref{prop:firsteigenvalues2}, it follows that $a_0^\odd=d$ is the first degeneracy instant for $f_\odd(a,0)=0$. As explained in the proof of Proposition~\ref{prop:firsteigenvalues2}, if $a=a_0^\odd=d$, then rotations in the $(x_1,x_4)$-plane induce \emph{isometries} of $\Omega_a$ that fix $O$ and act transitively on $\partial \Omega_a$; in particular, $f_\odd(d,s)=0$ for all $s\in\mathcal I$.
\end{proof}

The following result is the culmination of the spectral analysis above:

\begin{proposition}\label{prop:mainoscresult}
The spectra of $(\mathcal P_a)_\even$ and $(\mathcal P_a)_\odd$ 
are intertwined, that is, for all $a>0$, $b=c>0$, and $d>0$,
\begin{equation*}
\lambda_0^\even<\lambda_0^\odd<\lambda_1^\even<\lambda_1^\odd<\dots<\lambda_n^\even<\lambda_n^\odd<\lambda_{n+1}^\even<\lambda_{n+1}^\odd<\dots,
\end{equation*}
and each one of these eigenvalues is negative if $a>0$ is chosen sufficiently large. More precisely,
$\lambda_0^\even(a)<0$ for all $a>0$, and
\begin{enumerate}[\rm (i)]
\item if $n\in\N$, then $\lambda_n^\even(a_n^\even)=0$  and $\lambda_n^\even(a)<0$ for all $a>a_n^\even$;\smallskip
\item if $n\in\N_0$, then $\lambda_n^\odd(a_n^\odd)=0$  and $\lambda_n^\odd(a)<0$ for all $a>a_n^\odd$.
\end{enumerate}
In particular, the degeneracy instants 
$(a_n^\even)_{n\geq 1}$ and $(a_n^\odd)_{n\geq 0}$ are also intertwined:
\begin{equation*}
d=a_0^\odd < a_1^\even < a_1^\odd < a_2^\even < a_2^\odd< \dots < a_n^\even < a_n^\odd < a_{n+1}^\even < a_{n+1}^\odd <  \dots
\end{equation*}
\end{proposition}

\begin{center}
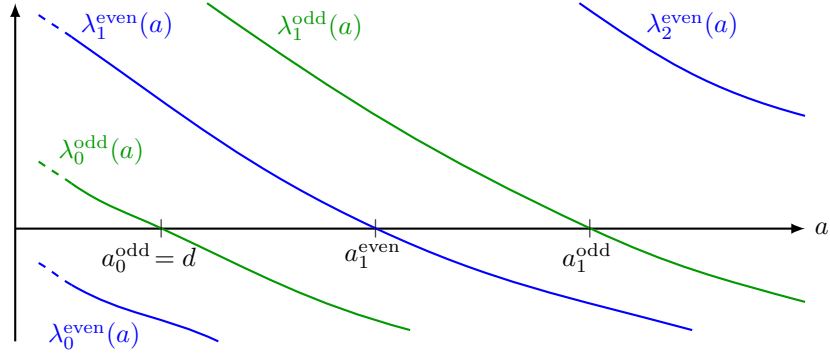
\begin{figure}[!ht]
\begin{tikzpicture}[scale=1.5]
    \draw [->,thick,-latex] (0,1) -- (7,1) node [right] {$a$};
    \draw [->,thick,-latex] (0,0) -- (0,3);
    \draw [-,thick,blue,dashed] (0.5,0.5) to[out=150,in=-30] (0.2,0.7);
    \draw [-,thick,blue] (0.5,0.5) to[out=-30,in=155] (1.8,0);
    \draw (0.7,0.05) node {\color{blue} $\lambda_0^\even(a)$};
    \draw (0.3,1.7) node [right] {\color{black!40!green} $\lambda_0^\odd(a)$};
    \draw [-,thick,black!40!green,dashed] (0.5,1.4) to[out=145,in=-35] (0.2,1.6);
    \draw [-,thick,black!40!green] (0.5,1.4) to[out=-35,in=155] (1.3,1);
    \draw [-,thick,black!40!green] (3.5,0.1) to[out=165,in=-25] (1.3,1);
    \draw (1.3,1) node {$_|$};
    \draw (1.3,0.95) node [below] {$a_0^{\odd}\!=d\quad$};
    \draw (0.5,2.8) node [right] {\color{blue} $\lambda_1^\even(a)$};
    \draw [-,thick,blue,dashed] (0.5,2.7) to[out=145,in=-35] (0.2,2.9);
    \draw [-,thick,blue] (0.5,2.7) to[out=-35,in=155] (3.2,1);
    \draw [-,thick,blue] (6,0.1) to[out=165,in=-25] (3.2,1);
    \draw (3.2,1) node {$_|$};
    \draw (3.2,0.95) node [below] {$a_1^{\even}\;$};
    \draw (2.7,2.8) node {\color{black!40!green} $\lambda_1^\odd(a)$};
    \draw [-,thick,black!40!green] (1.7,3) to[out=-35,in=155] (5.1,1);
    \draw [-,thick,black!40!green] (7,0.35) to[out=165,in=-25] (5.1,1);
    \draw [-,thick,blue] (5,3) to[out=-35,in=165] (7,2);
    \draw (5.1,1) node {$_|$};
    \draw (5.1,0.95) node [below] {$a_1^{\odd}\;$};
    \draw (6,2.8) node {\color{blue} $\lambda_2^\even(a)$};
\end{tikzpicture}
\caption{Schematic representation of eigenvalues of $(\mathcal P_a)_\even$ and $(\mathcal P_a)_\odd$ as $a>0$ varies, for fixed $b=c>0$ and $d>0$.}\label{fig:eigenvalues}
\end{figure}
\end{center}
\vspace{-0.5cm}

\begin{proof}
Fix $a>0$, and consider the following family of curves in the projective line:
\begin{equation}\label{eq:new-theta}
\phantom{(z,\lambda) [0,1)}
\Theta_{a}(\lambda)=\big[p_a(0)u_{a,\lambda}'(0) : u_{a,\lambda}(0)\big]\in\R P^1, \;\quad \text{for all }\lambda \in \R,
\end{equation}
where $u_{a,\lambda}(z)$ is defined in Proposition~\ref{prop:r-asolution}, and $[x_0 : x_1]$ are the usual homogeneous projective coordinates on $\R P^1$.
Consider the points
\begin{equation*}
\binfty=[0 : 1]\in \R P^1, \qquad \text{and} \qquad \mathbf 0=[1 : 0]\in \R P^1,
\end{equation*}
and note that $\lambda\in\R$ is equal to $\lambda_n^\even(a)$ or $\lambda_n^\odd(a)$ for some $n\in\N_0$ if and only if $\Theta_{a}(\lambda)$ is equal to $\binfty$ or $\mathbf 0$, respectively.
From \eqref{eq:lastpartialintegr}, we see that 
\begin{equation*}
\left(\tfrac{\mathrm d}{\mathrm d\lambda}\left(p_a(0)u_{a,\lambda}'(0)\right)\right)u_{a,\lambda}(0)-p_a(0)u_{a,\lambda}'(0)\left(\tfrac{\mathrm d}{\mathrm d\lambda}u_{a,\lambda}(0)\right)=1,
\end{equation*}
hence the curve \eqref{eq:new-theta} rotates clockwise, i.e., 
the line $\mathds R\cdot\big(p_a(0)u_{a,\lambda}'(0) , u_{a,\lambda}(0)\big)$ rotates clockwise in $\R^2$ as $\lambda$ increases. 
Therefore, \eqref{eq:new-theta} does not pass consecutively through one of 
$\binfty$ or $\mathbf 0$ without passing through the other, so the eigenvalues of $(\mathcal P_a)_\even$ and $(\mathcal P_a)_\odd$ are intertwined.
From Corollary~\ref{cor:even<odd}, the lowest among all eigenvalues of $(\mathcal P_a)_\even$ and $(\mathcal P_a)_\odd$ is $\lambda_0^\even(a)$, so the intertwining is as claimed.

From \eqref{eq:minmaxpaqa}, the quadratic form $Q_a$ is negative-definite on the subspace of $V_\even$ spanned by the constant function $v\equiv 1$. Thus, by Proposition~\ref{thm:spectralthsingular}, the eigenvalue $\lambda_0^\even(a)$ is negative for all $a>0$.
Meanwhile, $\lambda_0^\odd(a)$ crosses zero at $a=a_0^\odd=d$, by Corollary~\ref{cor:firstdegen}.
This implies claims (i) and (ii), as the sequences $a_n^\even$ and $a_n^\odd$  are unbounded by Proposition~\ref{prop:degsequences}, and all $\lambda_n^\even(a)$ and $\lambda_n^\odd(a)$ are decreasing by Proposition~\ref{prop:negdereigenv} and cannot cross each other. 
Note that the $n$-th degeneracy instants $a_n^\even$ and $a_n^\odd$ are intertwined as stated, since they correspond to values of $a$ at which the $n$-th eigenvalue $\lambda^\even_n(a)$ or $\lambda^\odd_n(a)$ crosses zero, and they do so with corresponding eigenfunctions $v_{a}\colon[0,1]\to\R$, that have exactly $n$ zeros in $(0,1)$.
(Enforcing these convenient self-indexing properties is the reason why the sequence $(a_n^\even)_{n\geq1}$ is indexed with $n\in\N$, while $(a_n^\odd)_{n\geq0}$ is indexed with $n\in\N_0$.)
\end{proof}

Henceforth, it will be convenient to consider the monotonic reindexing of the union of the sequences $(a_n^\odd)_{n\geq0}$ and $(a_n^\even)_{n\geq1}$, which by Proposition~\ref{prop:mainoscresult} is
\begin{equation}\label{eq:am}
    a_m:=\begin{cases}
        a^\even_{n}, & \text{ if } m=2n,\\
        a^\odd_{n}, & \text{ if } m=2n+1,
    \end{cases} \qquad m\in\N,
\end{equation}
i.e., $a_1=a_0^\odd=d$, $a_2=a_1^\even$, $a_3=a_1^\odd$, $a_4=a_2^\even$, and so on.

\subsection{Growth estimates}
We now provide a simple (and rather rough) estimate on the growth of the sequence $(a_m)_{m\geq1}$ defined in equation \eqref{eq:am}.

\begin{proposition}\label{prop:growth}
The sequence $(a_m)_{m\geq1}$ has the following asymptotic behavior:
\begin{equation}\label{eq:limsup_an}
\limsup\limits_{m\to\infty} \frac{a_m}{m} \leq 2d.
\end{equation}
\end{proposition}

\begin{proof}
From \eqref{eq:am}, it suffices to show that 
\begin{equation}\label{eq:limevenodd}
\limsup\limits_{n\to\infty} \frac{a^\even_n}{n} \leq 4d \quad \text{and} \quad \limsup\limits_{n\to\infty} \frac{a^\odd_n}{n} \leq 4d.    
\end{equation}

Consider the piecewise affine test functions $\xi_{\alpha,\delta,\varepsilon}\colon [0,1]\to[0,1]$ defined in the proof of Proposition~\ref{prop:degsequences}, with $\varepsilon=2\delta$. Using \eqref{eq:minmaxpaqa}, if $a>b$ and $b\delta<d$, then:
\begin{align*}
Q_a(\xi_{\alpha,\delta,2\delta}) &< \int_{\alpha}^{\alpha+\delta} p_a\,(\xi_{\alpha,\delta,2\delta}')^2 \,\dd z  + \int_{\alpha+3\delta}^{\alpha+4\delta} p_a\,(\xi_{\alpha,\delta,2\delta}')^2 \,\dd z  
+\int_{\alpha+\delta}^{\alpha+3\delta} q_a\,(\xi_{\alpha,\delta,2\delta})^2 \,\dd z \\
&< \frac{2}{\delta}\max_{z\in[0,1]} p_a(z)+2\delta\max_{z\in[0,1]}q_a(z) 
= \frac{2}{a\delta} -\frac{2(a^2+b^2)\delta}{a d^2}\leq 0,
\end{align*}
provided $a\geq\sqrt{\frac{d^2}{\delta^2}-b^2}$. Given $N\in\N$, set $\delta=\frac{1}{4N}$ and $\alpha_j=\frac{j}{N}$, $0\leq j< N$, so that at most one among $\xi_{\alpha_j,\delta,2\delta}$ can be nonzero at any given $z\in[0,1]$, i.e., the functions $\xi_{\alpha_j,\delta,2\delta}$, $0\leq j<N$, have pairwise disjoint support. By the above and standard density arguments, the quadratic form $Q_a$ is negative-definite on an $N$-dimensional subspace of $V_\even\cap V_\odd$, provided that $a\geq \sqrt{16N^2d^2-b^2}$. Thus, \eqref{eq:limevenodd} now follows 
from Propositions~\ref{thm:spectralthsingular} and \ref{prop:mainoscresult}, as $\lim\limits_{N\to+\infty}\frac{\sqrt{16N^2d^2-b^2}}{N}=4d$.
\end{proof}

\section{\texorpdfstring{Arbitrarily many nonplanar minimal $2$-spheres}{Arbitrarily many nonplanar minimal spheres}}\label{sec:main}

In this section, we prove Theorem \ref{mainthmA} in the Introduction by analyzing bifurcation branches of solutions to $f(a,s)=0$, where $f=f_\even$ or $f_\odd$, recall \eqref{eq:btriv}, \eqref{eq:bf}, Definition~\ref{def:fevenfodd}, and Propositions~\ref{prop:even_odd}, \ref{prop:char-degeneracies}, \ref{prop:degsequences}, and \ref{prop:mainoscresult}.  

First, we show that all degeneracy instants are bifurcation instants:

\begin{proposition}\label{prop:deg=bif}
All degeneracy instants $a_n^\even$, $n\geq1$, and $a_n^\odd$, $n\geq0$, satisfy the hypotheses of Theorem~\ref{thm:CR}, and are hence bifurcation instants. Therefore:
\begin{equation}\label{eq:bfeven-bfodd}
\mathfrak b(f_\even)=\big\{a_n^\even : n\in\N \big\}, \quad \text{ and }\quad \mathfrak b(f_\odd)=\big\{a_n^\odd : n\in\N_0\big\}.
\end{equation}
\end{proposition}

\begin{proof}
Hypothesis (i) of Theorem~\ref{thm:CR} follows from from Propositions~\ref{prop:char-degeneracies} and \ref{prop:degsequences}. From Proposition~\ref{prop:char-degeneracies} and \eqref{eq:SL-boundary-0}, we have that $\frac{\partial^2 f}{\partial a\,\partial s}(a_*,0)$ is a constant multiple of $\frac{\partial}{\partial a} u'_{a,0}(0)\big\vert_{a=a_n^\even}$ or $\frac{\partial}{\partial a}u_{a,0}(0)\big\vert_{a=a_n^\odd}$, according to $f=f_\even$ or $f_\odd$, respectively. The latter are nonzero by Lemma~\ref{lemma:deranonzero}, so hypothesis (ii) also holds. Together with the fact that every bifurcation instant is a degeneracy instant, this proves \eqref{eq:bfeven-bfodd}.
\end{proof}

\begin{definition}\label{def:bneven-bnodd}
We denote by $\mathfrak B_n^\even\subset f_\even^{-1}(0)$ and $\mathfrak B_n^\odd\subset f_\odd^{-1}(0)$ the bifurcation branches $\mathcal B_{a_*}$ issuing from the bifurcation point $(a_*,0)\in\mathcal B_\triv$, according to whether $a_*=a_n^\even$, or $a_*=a_n^\odd$, cf.~ Proposition~\ref{prop:deg=bif} and definitions in Section~\ref{sub:biftheory}. 
Collectively, we denote by $\mathfrak B_m$, $m\geq1$, the bifurcation branches corresponding to $a_m$, defined in \eqref{eq:am}, i.e., $\mathfrak B_1=\mathfrak B_0^\odd$, $\mathfrak B_2=\mathfrak B_1^\even$, $\mathfrak B_3=\mathfrak B_1^\odd$, $\mathfrak B_4=\mathfrak B_2^\even$ etc.
\end{definition}

\begin{remark}\label{rem:beven-bodd-reflections}
It follows from Remarks~\ref{rem:feven-even_fodd-odd} and \ref{rem:zeros-evenodd}, and Proposition~\ref{prop:deg=bif}, that each $\mathfrak B_m$ is invariant under the reflection $(a,s)\mapsto(a,-s)$ in $(0,+\infty)\times\mathcal I$.
\end{remark}

Note that, by Corollary~\ref{cor:firstdegen}, there is a trivial odd bifurcation branch given by
\begin{equation}\label{eq:b0odd}
\mathfrak B_1=\mathfrak B_0^\odd=\big\{(a,s): a=d, \, s\in \mathcal I\big\}.
\end{equation}
Aside from the above, all other bifurcation branches $\mathfrak B_m$ consist of points $(a,s)$ where $s\in\mathcal I$ is bounded away from the endpoints $\partial\mathcal I=\{\pm\frac{\pi}{2}\}$ if $a$ remains in compact subsets of $(0,+\infty)$; namely, we have the following result:

\begin{lemma}\label{lemma:away-from-bdy}
There exists a positive function $\varepsilon \colon (d,+\infty)\to  \R$ locally bounded away from $0$, such that if $a>d$ and $f(a,s)=0$, where $f=f_\even$ or $f_\odd$, then\linebreak $s\in\left[-\frac{\pi}{2}+\varepsilon(a),\frac{\pi}{2}-\varepsilon(a)\right]$. Thus, the restriction 
to $\big(f_\even^{-1}(0)\cup f_\odd^{-1}(0)\big)\cap (d,+\infty)\times \mathcal I$
of the projection $p\colon (d,+\infty)\times\mathcal I\to(d,+\infty)$ onto the first factor  is a proper map.
\end{lemma}

\begin{proof}
From Propositions~\ref{prop:negdereigenv} and \ref{prop:firsteigenvalues2}, we have $\lambda_0^\odd(a)>0$ for all $0<a<d$. Thus, by Proposition~\ref{prop:char-degeneracies} and Definition~\ref{def:fevenfodd}, if $0<a<d$, then the horizontal geodesic $\gamma_\hor$ in $\Omega_a$ is nondegenerate, hence isolated. 
Exchanging the roles of $\gamma_\hor$ and $\gamma_\ver$, hence of the parameters $a$ and $d$, it follows by the same arguments that $\gamma_\ver$ is isolated if $a>d$. In other words, the subset $\big(f_\even^{-1}(0)\cup f_\odd^{-1}(0)\big)\cap (d,+\infty)\times \mathcal I$ is bounded away from $(d,+\infty)\times \partial \mathcal I$, which proves the first assertion. The conclusion that $p\colon (d,+\infty)\times\mathcal I\to(d,+\infty)$ restricted to this subset is a proper map now follows from the compactness result of Choi--Schoen~\cite{choi-schoen}, also cf.~\cite[p.~163]{white2}.
\end{proof}

Thus, for fixed $a>d$, all zeros of the real analytic functions $\mathcal I\ni s\mapsto f(a,s)$, $f=f_\even$ or $f_\odd$, lie in a compact subinterval of $\mathcal I$, which proves the following:

\begin{corollary}\label{thm:finnumbs}
For any fixed $a>d$, we have that $\#\big(f^{-1}_\even(0)\cap\big(\{a\}\times \mathcal I\big)\big)<+\infty$ and $\#\big(f^{-1}_\odd(0)\cap\big(\{a\}\times \mathcal I\big)\big)<+\infty$.
\end{corollary}

We are now ready to prove the main result of this section, illustrated in Figure~\ref{fig:branches}. 

\begin{theorem}\label{thm:mainthm-bifbranches}
The bifurcation branches $\mathfrak B_m$, $m\geq1$, are 
noncompact and pairwise disjoint subsets that disconnect $(0,+\infty)\times \mathcal I$. If $m\geq2$, then $p(\mathfrak B_m)\supset [a_m,+\infty)$, where $p\colon (d,+\infty)\times\mathcal I\to (d,+\infty)$ is the projection on the first factor.
In particular, if $a>a_m$, then the segment $\{a\}\times\mathcal I$ intersects at least $m-1$ bifurcation branches.
\end{theorem}

\begin{proof}
First, we claim that the function $Z\colon f^{-1}(0)\setminus\mathcal B_\triv\to\N_0$ given by
\begin{equation}\label{eq:z}
Z(a,s) := \#\big\{z\in (0,1) : \sigma_{x_4}(a,s,z)=0 \big\}
\end{equation}
is a discrete-valued invariant for the equation $f(a,s)=0$, where $f=f_\even$ or $f_\odd$, recall Definitions~\ref{def:invariant} and \ref{def:fevenfodd}, and \eqref{eq:sigmas}.
Indeed, note that the reparametrization of the geodesic $\gamma_{a,s}$ given by $z\mapsto \sigma(a,s,z)$,  $a>0$, $s\in\mathcal I\setminus\{0\}$, intersects $\gamma_\hor$ precisely at those $z$ where $\sigma_{x_4}(a,s,z)=0$, and every such intersection is transversal, i.e., $z\mapsto \sigma_{x_4}(a,s,z)$ only has simple zeros. 
If $(a,s)\in f_\even^{-1}(0)\setminus\mathcal B_\triv$, then
$\gamma_{a,s}$ is an even geodesic and $s\neq 0$, so $\sigma_{x_4}(a,s,0)\neq0$ and $\lim_{z\nearrow1}\sigma_{x_4}(a,s,z)\neq0$. Similarly, if $(a,s)\in f_\odd^{-1}(0)\setminus\mathcal B_\triv$, then
$\gamma_{a,s}$ is an odd geodesic and $s\neq 0$, so $\sigma_{x_4}(a,s,0)=0$ and $\lim_{z\nearrow1}\sigma_{x_4}(a,s,z)\neq0$. In both cases, it follows that \eqref{eq:z} is locally constant on $f^{-1}(0)\setminus\mathcal B_\triv$, and hence a well-defined discrete-valued invariant.

By Proposition~\ref{prop:deg=bif}, the hypotheses of Theorem~\ref{thm:CR} hold at each $a_*\in\mathfrak b(f)$, with $f=f_\even$ and $f_\odd$. Moreover, by Remarks~\ref{rem:zeros-evenodd} and \ref{rem:feven-even_fodd-odd}, 
the bifurcation branch $\mathcal B_{a_*}$ in a sufficiently small neighborhood of $(a_*,0)$ can be parametrized as a curve $(a(s),s)$, $s\in (-\varepsilon,\varepsilon)$, where $a(s)$ is a real analytic even function with $a(0)=a_*$. In particular, $a'(0)=0$. Recalling \eqref{eq:v_a}, this implies that, for all $z\in[0,1)$,
\begin{equation}
\textstyle
\frac{\partial}{\partial s}\sigma_{x_4}(a(s),s,z)\big|_{s=0}=\frac{\partial}{\partial a}\sigma_{x_4}(a_*,0,z)\,a'(0)+\frac{\partial}{\partial s}\sigma_{x_4}(a_*,s,z)\big|_{s=0}=v_{a_*}(z).
\end{equation}
Thus, the value assumed by the locally constant function $Z$ at points $(a,s)\in \mathcal B_{a_*}$ with $s\neq0$ sufficiently small is $Z(a,s)=z^\pm(a_*)=\#\{z\in (0,1) : v_{a_*}(z)=0\}$, cf.~\eqref{eq:zpm}. By Propositions~\ref{thm:spectralthsingular} and \ref{prop:char-degeneracies}, the latter is equal to $n$ if $a_*=a_n^\even$ or $a_n^\odd$, hence $z^\pm(a_n^\even)=z^\pm(a_n^\odd)=n$. Therefore, $z^\pm\colon \mathfrak b(f)\to\N_0$ is injective if $f=f_\even$ or $f_\odd$, see~\eqref{eq:bfeven-bfodd}.

We now verify the remaining hypotheses of Proposition~\ref{prop:disjoint_noncompact} (in particular, those of Theorem~\ref{thm:rabinowitz}) in both cases $f=f_\even$ and $f_\odd$, to conclude that the bifurcation branches $\mathfrak B_m$  are noncompact and pairwise disjoint.
Recall the trivial odd bifurcation branch $\mathfrak B_1=\mathfrak B_0^\odd$ given by \eqref{eq:b0odd}, and note that none of $\mathfrak B_n^\even$ nor of $\mathfrak B_n^\odd$, with $n\geq1$, can intersect it. Indeed, $z^\pm(a_0^\odd)=0$ and $z^\pm(a_*)\geq1$ for all $a_*\in\mathfrak b(f)$, with $f=f_\even$ or $f_\odd$. Thus, we shall disregard the first bifurcation instant $a_0^\odd$ and assume $a>d$.
By Corollary~\ref{thm:finnumbs}, the maps $s\mapsto a(s)$ with $a(0)=a_*$ whose graph $(a(s),s)$ parametrizes $\mathcal B_{a_*}$ near $(a_*,0)$ are not constant, so hypothesis (1) in Theorem~\ref{thm:rabinowitz} holds. Moreover, the restrictions of $p\colon (d,+\infty)\times\mathcal I\to(d,+\infty)$ to $f_\even^{-1}(0)$ and to $f_\odd^{-1}(0)$ are proper maps by Lemma~\ref{lemma:away-from-bdy}. 
So, Proposition~\ref{prop:disjoint_noncompact} applies, and we conclude that the bifurcation branches in the families $\mathfrak B_n^\even$ and $\mathfrak B_n^\odd$ are noncompact and pairwise disjoint.
Moreover, pairs of bifurcation branches from distinct families also do not intersect, because $(f_\even^{-1}(0)\setminus \mathcal B_\triv)\cap (f_\odd^{-1}(0)\setminus \mathcal B_\triv)=\emptyset$.

By Proposition~\ref{prop:mainoscresult}, the bifurcation branch that contains points $(a,s)$ with the smallest possible $a>d$ is $\mathfrak B_1^\even$. Since $\mathfrak B_1^\even\subset f_\even^{-1}(0)$ is closed and contains only points $(a,s)$ with $|s|\leq \frac{\pi}{2}-\varepsilon(a)<\frac{\pi}{2}$ by Lemma~\ref{lemma:away-from-bdy}, it follows that there exists $\delta>0$ such that $p(\mathfrak B_1^\even)\subset [d+\delta,+\infty)$. Thus, also $p(\mathfrak B_n^\even)\subset [d+\delta,+\infty)$ and $p(\mathfrak B_n^\odd)\subset [d+\delta,+\infty)$ for all $n\geq1$. 
Since $\mathfrak B_m$ are connected, so are $p(\mathfrak B_m)$, for all $m\geq1$. By definition, $a_m\in p(\mathfrak B_m)$. If some $p(\mathfrak B_m)$, $m\geq2$, was bounded, then it would have a supremum $\overline a<\infty$, and hence be contained in $[d+\delta,\overline a]$. Since the restriction of $p$ to $f^{-1}(0)$ is proper, this would imply that the compact set $p^{-1}([d+\delta,\overline a])\cap f^{-1}(0)$ contains a closed but noncompact subset. Thus, it follows that $p(\mathfrak B_m)\supset [a_m,+\infty)$ for all $m\geq2$.
It also follows from the above that each $\mathfrak B_m$ disconnects the strip $(d,+\infty)\times \mathcal I$, since it contains a curve of the form \eqref{eq:bif-curve-extended} satisfying \eqref{dich:bdy} in Theorem~\ref{thm:rabinowitz}, and $\lim_{t\to\pm\infty} (a(t),s(t))\notin (d,\infty)\times\partial \mathcal I$ by Lemma~\ref{lemma:away-from-bdy}.
Finally,  $\#\big(f^{-1}(0)\cap\big(\{a\}\times \big(0,\frac{\pi}{2}\big)\big)\big)\geq n$ if $a\geq a_n^\even$ or $a_n^\odd$ according to $f=f_\even$ or $f_\odd$. Taking into account $\mathfrak B_1=\mathfrak B_0^\odd$, given by \eqref{eq:b0odd}, and recalling \eqref{eq:am}, we conclude that $\{a\}\times\mathcal I$ intersects at least $m-1$ bifurcation branches if $a>a_m$.
\end{proof}

\begin{center}
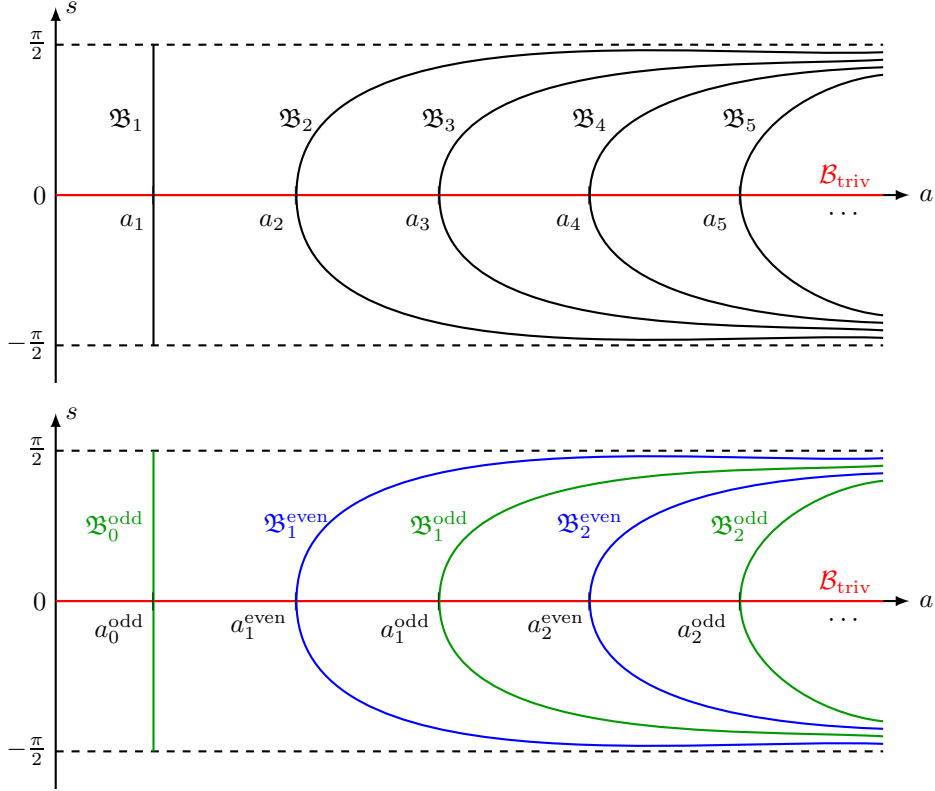
\begin{figure}[!ht]
\begin{tikzpicture}[scale=1]
    \draw [->,thick,-latex] (0,-2.5) -- (0,2.5) node [right] {$s$};
    \draw [->,thick,-latex] (0,0) -- (11.35,0) node [right] {$a$};
    
    \draw (0,2) node [left] {$\frac{\pi}{2}$};
    \draw (0,0) node [left] {$0$};
    \draw (0,-2) node [left] {$-\frac{\pi}{2}$};
    \draw [-,thick,dashed] (0,-2) -- (11,-2);
    \draw [-,thick,dashed] (0,2) -- (11,2);

    \draw (10.5,0.25) node {$\color{red}\mathcal B_\triv$};
    \draw [-,thick,red] (0,0) -- (11,0);

    \draw (1.3,1) node [left] {$\mathfrak B_1$};
    \draw [-,thick] (1.3,-2) -- (1.3,2);
    \draw (1.3,0) node {$_|$};
    \draw (1.3,-.05) node [below] {$a_1\phantom{=d}$};

    \draw (3.2,1) node {$\mathfrak B_2$};
    \draw [-,thick] (3.2,0) to[out=90,in=182,looseness=0.8] (11,1.9);
    \draw [-,thick] (3.2,0) to[out=-90,in=-182,looseness=0.8] (11,-1.9);
    \draw (3.2,0) node {$_|$};
    \draw (3.2,-.05) node [below] {$a_2\;\phantom{=d}$};
    
    \draw (5.1,1) node {$\mathfrak B_3$};
    \draw [-,thick] (5.1,0) to[out=90,in=183,looseness=0.8] (11,1.8);
    \draw [-,thick] (5.1,0) to[out=-90,in=-183,looseness=0.8] (11,-1.8);
    \draw (5.1,0) node {$_|$};
    \draw (5.1,-0.05) node [below] {$a_3\phantom{=d}$};

    \draw (7.1,1) node {$\mathfrak B_4$};
    \draw [-,thick] (7.1,0) to[out=90,in=182,looseness=0.85] (11,1.7);
    \draw [-,thick] (7.1,0) to[out=-90,in=-182,looseness=0.85] (11,-1.7);
    \draw (7.1,0) node {$_|$};
    \draw (7.1,-0.05) node [below] {$a_4\phantom{=d}$};

    \draw (9.1,1) node {$\mathfrak B_5$};
    \draw [-,thick] (9.1,0) to[out=90,in=186,looseness=0.9] (11,1.6);
    \draw [-,thick] (9.1,0) to[out=-90,in=-186,looseness=0.9] (11,-1.6);
    \draw (9.1,0) node {$_|$};
    \draw (9.1,-0.05) node [below] {$a_5\;\phantom{=d}$};
    \draw (10.5,-0.05) node [below] {$\cdots$};
\end{tikzpicture}

\medskip

\begin{tikzpicture}[scale=1]
    \draw [->,thick,-latex] (0,-2.5) -- (0,2.5) node [right] {$s$};
    \draw [->,thick,-latex] (0,0) -- (11.35,0) node [right] {$a$};
    
    \draw (0,2) node [left] {$\frac{\pi}{2}$};
    \draw (0,0) node [left] {$0$};
    \draw (0,-2) node [left] {$-\frac{\pi}{2}$};
    \draw [-,thick,dashed] (0,-2) -- (11,-2);
    \draw [-,thick,dashed] (0,2) -- (11,2);

    \draw (10.5,0.25) node {$\color{red}\mathcal B_\triv$};
    \draw [-,thick,red] (0,0) -- (11,0);

    \draw (1.3,1) node [left] {$\color{black!40!green}\mathfrak B_0^\odd$};
    \draw [-,thick,black!40!green] (1.3,-2) -- (1.3,2);
    \draw (1.3,0) node {$_|$};
    \draw (1.3,-.05) node [below] {$a_0^{\odd}\phantom{=d\quad}$};
        
    \draw (3.2,1) node {$\color{blue}\mathfrak B_1^\even$};
    \draw [-,thick,blue] (3.2,0) to[out=90,in=182,looseness=0.8] (11,1.9);
    \draw [-,thick,blue] (3.2,0) to[out=-90,in=-182,looseness=0.8] (11,-1.9);
    \draw (3.2,0) node {$_|$};
    \draw (3.2,-.05) node [below] {$a_1^{\even}\;\phantom{=d\quad}$};
    
    \draw (5.1,1) node {$\color{black!40!green}\mathfrak B_1^\odd$};
    \draw [-,thick,black!40!green] (5.1,0) to[out=90,in=183,looseness=0.8] (11,1.8);
    \draw [-,thick,black!40!green] (5.1,0) to[out=-90,in=-183,looseness=0.8] (11,-1.8);
    \draw (5.1,0) node {$_|$};
    \draw (5.1,-0.05) node [below] {$a_1^{\odd}\phantom{=d\quad}$};

    \draw (7.1,1) node {$\color{blue}\mathfrak B_2^\even$};
    \draw [-,thick,blue] (7.1,0) to[out=90,in=182,looseness=0.85] (11,1.7);
    \draw [-,thick,blue] (7.1,0) to[out=-90,in=-182,looseness=0.85] (11,-1.7);
    \draw (7.1,0) node {$_|$};
    \draw (7.1,-0.05) node [below] {$a_2^{\even}\phantom{=d\quad}$};

    \draw (9.1,1) node {$\color{black!40!green}\mathfrak B_2^\odd$};
    \draw [-,thick,black!40!green] (9.1,0) to[out=90,in=186,looseness=0.9] (11,1.6);
    \draw [-,thick,black!40!green] (9.1,0) to[out=-90,in=-186,looseness=0.9] (11,-1.6);
    \draw (9.1,0) node {$_|$};
    \draw (9.1,-0.05) node [below] {$a_2^{\odd}\;\phantom{=d\quad}$};
    \draw (10.5,-0.05) node [below] {$\cdots$};
    
\end{tikzpicture}
\caption{Schematic illustration of trivial branch $\mathcal B_\triv$ in red; and bifurcation branches $\mathfrak B_m$, $m\geq1$, in the upper picture. In the lower picture, these branches are individuated as $\mathfrak B_n^\even$ and $\mathfrak B_n^\odd$.}\label{fig:branches}
\end{figure}
\end{center}

\begin{remark}\label{rem:geom-interp}
From the proof of Theorem~\ref{thm:mainthm-bifbranches}, namely the computation of \eqref{eq:z}, it follows that if $(a,s)\in\mathfrak B_m$, then $\gamma_{a,s}$ intersects $\gamma_\hor$ transversely at $m$ points, and hence the embedded minimal $2$-sphere $\Pi^{-1}(\gamma_{a,s})$ in $E(a,b,b,d)$ intersects $\Sigma_4(a)$ transversely at $m$ circles parallel to the $(x_2,x_3)$-plane. In particular, cutting along $\Sigma_4(a)$ disconnects it into $m-1$ annuli, and $2$ disks. It is straightforward to verify that these minimal $2$-spheres are pairwise noncongruent for different values of $m$.

The above is in stark contrast with the \emph{two-piece property} of minimal surfaces in the round $3$-sphere, which are disconnected into \emph{two} connected components when cut along any planar $2$-sphere (i.e., an equator), see Ros~\cite{Ros95}.
\end{remark}

\begin{proof}[Proof of Theorem~\ref{mainthmA}]
From Definition~\ref{def:fevenfodd} and Proposition~\ref{prop:even_odd}, 
the number $N(a)$ of noncongruent nonplanar embedded minimal $2$-spheres in $E(a,b,b,d)$ satisfies:
\begin{equation}\label{eq:na-lowerbound}
N(a)\geq \#\left\{m \geq 2 : \mathfrak B_m \cap \left(\{a\}\times \left(0,\tfrac\pi2\right) \right)\neq \emptyset \right\},
\end{equation}
see also Remark~\ref{rem:geom-interp}. 
Consider the nondecreasing function
\begin{equation*}
N_{\mathfrak b}\colon (0,+\infty)\longrightarrow\N_0, \quad 
N_{\mathfrak b}(a):= 
\#\{m\geq 2 : a>a_m\}.
\end{equation*}
By Theorem~\ref{thm:mainthm-bifbranches}, 
the right-hand side of \eqref{eq:na-lowerbound} is $\geq N_\mathfrak b(a)$, hence $N(a)\geq N_\mathfrak b(a)$. (Recall that $a_1=a_0^\odd=d\in\mathfrak b(f_\odd)$ gives rise to the branch \eqref{eq:b0odd}, so this first bifurcation instant does not contribute to $N(a)$, nor to $N_\mathfrak b(a)$.)
Thus, we have:
\begin{equation}\label{eq:proofliminf1}
\liminf_{a\to+\infty} \frac{N(a)}{a} \geq  \liminf_{a\to+\infty}  \frac{N_\mathfrak b(a)}{a}  = \lim_{n\to\infty} \frac{N_\mathfrak b(\alpha_n)}{\alpha_n}
\end{equation}
for some diverging sequence $(\alpha_n)_{n\in\N}$.
From the proof of Proposition~\ref{prop:growth}, it follows that if $n\in\N$ is sufficiently large, then $a \geq \sqrt{16n^2d^2-b^2}$ implies $N_\mathfrak b(a)\geq 2n-2$; because, prior to such value of $a$, there are at least $n-1$ even bifurcation instants and $n-1$ odd bifurcation instants other than $a_1$. Choosing an increasing function $k\colon \N\to\N$ with $\sqrt{16\,k(n)^2d^2-b^2}\leq \alpha_n < \sqrt{16(k(n)+1)^2d^2-b^2}$ for all $n\in\N$, 
\begin{equation}\label{eq:proofliminf2}
\begin{aligned}
\lim_{n\to+\infty} \frac{N_\mathfrak b(\alpha_n)}{\alpha_n} &\geq \limsup_{n \to+\infty} \frac{N_\mathfrak b\big(\!\sqrt{16\,k(n)^2d^2-b^2}\big)}{\sqrt{16(k(n)+1)^2d^2-b^2}}\\
&\geq \limsup_{n\to+\infty} \frac{2k(n)-2}{\sqrt{16(k(n)+1)^2d^2-b^2}}=\frac{1}{2d}.
\end{aligned}
\end{equation}
The desired estimate \eqref{eq:liminfNa} on the growth of $N(a)$ follows from \eqref{eq:proofliminf1} and \eqref{eq:proofliminf2}.
\end{proof}

\begin{remark}\label{rem:bifa0}
A bifurcation analysis similar to the above for $\Sigma_4(a)=\Pi^{-1}(\gamma_\hor)$ as $a\nearrow+\infty$ can be carried out for $\Sigma_1(a)=\Pi^{-1}(\gamma_\ver)$ as $a\searrow 0$, and yields further nonplanar embedded minimal $2$-spheres in $E(a,b,b,d)$. In particular, the zero sets of $f_\even$ and $f_\odd$ contain points \emph{outside} the bifurcation branches $\mathfrak B_m$ issuing from $\mathcal B_\triv$, illustrated in Figure~\ref{fig:branches}; namely, points in bifurcation branches issuing from \emph{another} trivial branch $\big\{\big(a,\pm\frac\pi2\big):a>0\big\}$, corresponding to $\gamma_\ver$. 
The key step in this analysis is to show that the equivariant Morse index of $\Sigma_1(a)$ diverges as $a\searrow0$, which follows reasoning as in Proposition~\ref{prop:degsequences}, 
since the Jacobi equation for radial functions on $\Sigma_1^+(a)$ is also of the form \eqref{eq:SL}, where $a$ and $d$ exchanged in \eqref{eq:pa-qa}.
\end{remark}

\section{\texorpdfstring{Asymptotic behavior and convergence as $a\nearrow+\infty$}{Asymptotic behavior as a goes to infinity}}\label{sec:asymptotic}

In this section, we prove Theorem~\ref{mainthmB} in the Introduction. 
To that end, we first analyze the limit of the Riemannian $2$-disk $\Omega_a$, and of its geodesics, as $a\nearrow+\infty$.

\begin{proposition}\label{thm:limitstrip}
The family of open Riemannian $2$-disks $\Omega_a=\big(S^3_{\pr}/\G,V^2\check \g_a\big)$ converges smoothly as $a\nearrow+\infty$ to the open Riemannian strip
\begin{equation}\label{eq:omega_infty}
\Omega_{\infty}:=\big( \R\times (-L,L),\, \eta(y)^2\big(\dd x^2+\dd y^2\big) \big),
\end{equation}
where $\eta(y)=2\pi b\cos\!\big(v(y)\big)$, and $v(y)$ is the inverse function of
\begin{equation}\label{eq:y}
y\colon \left[-\tfrac\pi2,\tfrac\pi2\right]\longrightarrow [-L,L], \quad y(v)=\int_0^v \sqrt{d^2\cos^2\xi +b^2\sin^2\xi}\;\dd \xi.
\end{equation}
\end{proposition}

\begin{proof}
The orbit space $(S^3/\G,\check\g_a)$ is isometric to the ellipsoidal hemisphere \eqref{eq:quotient}, and can hence be parametrized with $X(u,v)=\big(x_1(u,v),r(u,v),x_4(u,v)\big)$, where
\begin{equation}\label{eq:param_omega_a}
x_1(u,v)=a\cos u\sin v, \quad r(u,v)=b\cos v, \quad x_4(u,v)=d\sin u\sin v,
\end{equation}
and $u\in [0,\pi]$, $v\in[-\tfrac\pi2,\tfrac\pi2]$.
This parametrization leads to an intrinsic expression for $\check\g_a$, which can be computed from the Euclidean metric $\dd x_1^2+\dd r^2+\dd x_4^2$ and \eqref{eq:param_omega_a}. It is easy to see that the limit of $\check\g_a$ as $a\nearrow+\infty$ is the flat metric on the strip $\R\times [-L,L]$, namely
\begin{equation*}
\dd x^2+\dd y^2=\dd x_1^2+\big(d^2\cos^2 v+b^2\sin^2 v\big)\dd v^2,
\end{equation*}
where $x=x_1$, and 
$y=y(v)$ is given by \eqref{eq:y}, with $L=y(\tfrac\pi2)$. 
The conclusion now follows, recalling that the volume function \eqref{eq:vol-orbits} is given by $V=2\pi b\cos v$.
\end{proof}

Note that the boundary $\partial\Omega_a$, which corresponds to $r=0$, is the ellipse with semiaxes $a$ and $d$ given by the union of the arcs $\partial_\pm \Omega_a=\left\{X\!\left(u,\pm\tfrac\pi2\right) : 0\leq u\leq \pi\right\}$ that meet at the antipodal points $(\pm a,0,0)$ on the $x_1$-axis. In the limit as $a\nearrow+\infty$, these become the disjoint components $\partial_\pm\Omega_\infty=\R\times \{\pm L\}$ of the boundary $\partial \Omega_\infty$.

\begin{remark}
For all values of $a\in (0,+\infty]$, the length of the geodesic $\gamma_\ver$ in~$\Omega_a$, parametrized as $X\!\left(\tfrac\pi2,v\right)$ above, is equal to $2L$ with respect to the metric $\check\g_a$, but equal to $|\Sigma_1(a)|$ with respect to the metric $V^2\,\check \g_a$, since $\Pi^{-1}(\gamma_\ver)=\Sigma_1(a)$.
\end{remark}

A simple analysis similar to \cite[Sec.~2, 3]{hnr} for geodesics of \eqref{eq:omega_infty} leads to:

\begin{lemma}\label{lem:geods_omega_infty}
A maximal geodesic $\gamma$ in $\Omega_\infty$, see \eqref{eq:omega_infty}, is either a vertical segment $\{x_0\}\times (-L,L)$, for some $x_0\in\R$, or the horizontal geodesic $\mathds R\times\{0\}$, or else the graph of a surjective periodic function $y\colon\mathds R\to[-w_\gamma,w_\gamma]$, where $w_\gamma\in (0,L)$.
\end{lemma}

\begin{proof}
By Proposition~\ref{thm:limitstrip}, the conformally flat metric $\eta(y)^2(\mathrm dx^2+\mathrm dy^2)$ is such that $\eta\colon (-L,L)\to\R$ is an even positive function, strictly decreasing on $\left[0,L\right]$, and with simple zeros at $y=\pm L$. Under these conditions, given any unit speed geodesic $\gamma(t)=(x(t),y(t))$ in $\Omega_\infty$, 
since $\frac{\partial}{\partial x}$ is a Killing field, we have that $\eta(y(t))^2 x'(t)\equiv C_\gamma$, for some $C_\gamma\in\R$. 
Since $x'(t)=C_\gamma\eta(y(t))^{-2}$ and $\eta(y(t))^2x'(t)^2\le1$, we have $\vert x'(t)\vert\le\eta(y(t))^{-1}$, and therefore 
\begin{equation}\label{eq:boundsCgamma}
\vert C_\gamma\vert\le \eta(y(t))\le \eta(0)=2\pi b.
\end{equation} 
If $C_\gamma=0$, then $\gamma$ is a vertical segment. If $|C_\gamma|=2\pi b$, then $y(t)\equiv0$ and $\gamma$ is the $x$-axis $\R\times\{0\}$. So let us assume $0<C_\gamma<2\pi b$, the case $-2\pi b<C_\gamma<0$ being completely analogous.
By \eqref{eq:boundsCgamma}, $\gamma$ remains in the strip $\mathds R\times[-w_\gamma,w_\gamma]$, where $w_\gamma\in\left(0,L\right)$ and $\eta(w_\gamma)=C_\gamma$. In particular, $\gamma$ must be defined on the entire real line. Moreover, since $x'\ge C_\gamma/4\pi^2b^2>0$, we obtain that $x$ is a strictly increasing function, with $\lim\limits_{t\to\pm\infty}x(t)=\pm\infty$. 
This implies that $\gamma$ is the graph of a smooth function $y\colon\mathds R\to[-w_\gamma,w_\gamma]$, which we now show is periodic.

It is easy to prove that $y'(t)$ vanishes if $t=t_n$, where $(t_n)_n$ is an increasing unbounded sequence. Namely, if this were not the case, then $y'$ would not vanish on a half-line $\left[t_*,+\infty\right)$, and $\gamma\vert_{\left[t_*,+_\infty\right)} $ could then be parametrized by $y$ (assuming $y'>0$, the case $y'$ being totally analogous). But then, setting $y_*=y(t_*)$, we have
\begin{equation}\label{eq:improper}
x(y)-x(y_*)=\int_{y_*}^y\frac{C_\gamma}{\sqrt{\eta(s)^2-C_\gamma^2}}\,\mathrm ds\le\int_{y_*}^{w_\gamma}\frac{C_\gamma}{\sqrt{\eta(s)^2-C_\gamma^2}}\,\mathrm ds, 
\end{equation}
which gives a contradiction, because $x(y)$ is unbounded as $y$ grows, while the improper integral on the right-hand side of \eqref{eq:improper} is convergent, since the function $\eta^2-C_\gamma^2$ has a simple zero at $s=w_\gamma$. 
Note that the zeros of $y'$ correspond to points along $\gamma$ for which $y=\pm w_\gamma$, and hence $y'$ changes sign at every zero. The existence of zeros for $y'$ also proves that the function $y\colon\mathds R\to[-w_\gamma,w_\gamma]$, whose graph is $\gamma$, is surjective.
Now, choose an instant $t_1$ with $y'(t_1)=0$, define $t_2$ to be the minimal $t>t_1$ with $y(t_2)=y(t_1)$ (in particular, $y'(t_2)=0$), and set  $\Delta_\gamma=x(t_2)-x(t_1)$. Since translations in the $x$-directions are isometries, the curve $\widetilde\gamma(t)=\big(x(t)+\Delta_\gamma,y(t)\big)$ is a geodesic in $\Omega_\infty$, and satisfies $\widetilde\gamma(t_1)=\gamma(t_2)$ and $\widetilde\gamma'(t_1)=\gamma'(t_2)$. Therefore, $\widetilde\gamma(t)=\gamma(t+t_2-t_1)$ for all $t\in\mathds R$, so $y=y(x)$ is periodic, with period $\Delta_\gamma$, concluding the proof.
\end{proof}

The above leads us to define the \emph{period function} $\Delta\colon\left(-2\pi b,2\pi b\right)\to \R$, given by 
\begin{equation}\label{eq:form_period}
\Delta(c):=2\int_{-w_c}^{w_c}\frac{\vert c\vert}{\sqrt{\eta(s)^2-c^2}}\,\mathrm ds,    
\end{equation}
where $w_c\in\eta^{-1}(c)\cap (0,L)$. Indeed, by the proof of Lemma~\ref{lem:geods_omega_infty}, if $\gamma(t)=(x(t),y(t))$ is a maximal geodesic in $\Omega_\infty$ that is neither vertical nor horizontal, i.e., such that $\vert C_\gamma\vert\in\left(0,2\pi b\right)$, then the period of $y$ as a function of $x$~is $\Delta_\gamma=\Delta(C_\gamma)$.
By the Dominated Convergence Theorem, $\Delta(c)$ is continuous, and hence $\lim\limits_{c\to 0}\Delta(c)=0$.

\begin{proposition}\label{thm:finiteintersectionbounded}
Let $\gamma(t)=(x(t),y(t))$ be a unit speed geodesic in $\Omega_\infty$, with $x(0)=0$ and $C_\gamma \in (0,2\pi b)$. Let $T>0$ and $m\ge1$ be fixed, and assume that $\gamma\big([-T,T]\big)$ intersects the $x$-axis at most $m$ times. Then, for all $t\in[-T,T]$,
\begin{equation}\label{eq:boundsxnumzeros}
\big \vert x(t)\big\vert \le\big(1+\tfrac m2\big)\Delta_\gamma.
\end{equation}
\end{proposition}

\begin{proof}
By Lemma~\ref{lem:geods_omega_infty}, $\gamma$ is the graph of a function $y=y(x)$ that is periodic of period $\Delta_\gamma$. Intersections of $\gamma$ with the $x$-axis correspond to zeros of $y$, and on each interval $\left[\alpha,\alpha+\Delta_\gamma\right)$, the function $y$ has exactly two zeros, which implies  \eqref{eq:boundsxnumzeros}.
\end{proof}

In order to state the next result, recall that the Riemannian disk $\Omega_a$ is identified with the ellipsoidal hemisphere \eqref{eq:quotient}, contained in $\mathds R^3$, and endowed with a metric conformal to the induced Euclidean metric, parametrized by  $(x_1,r,x_4)$, as in \eqref{eq:param_omega_a}.

\begin{proposition}\label{prop:corgammaan-new}
Given $m\geq2$, let $(\alpha_k,s_k)_{k\in\N}$ be a sequence of points in $\mathcal B_m$ with $\alpha_k\nearrow+\infty$ as $k\to+\infty$. 
Let $\gamma_k\colon (-1,1)\to \Omega_{\alpha_k}$ be the affine reparametrization of $\gamma_{\alpha_k,s_k}$, with coordinates $\big(x_{1,k},r_k,x_{4,k}\big)$. In particular, 
 $x_{1,k}\colon (-1,1)\to \R$ is odd, while $x_{4,k}\colon (-1,1)\to \R$ is either odd or even, according to $m$ being odd or even. For all $1\leq\ell\leq m$, let $I_\ell:=\big(-1+\frac{2(\ell-1)}m,-1+\frac{2\ell}m\big)$, and let $\widehat\gamma_k\colon (-1,1)\to \Omega_\infty$ be the curve with components $\big(x_{1,k},x_{4,k}\big)$. Then, for each $1\leq \ell\leq m$, the sequence $\left(\widehat\gamma_k\vert_{I_\ell}\right)_{k}$ converges smoothly, up to subsequences, to the vertical geodesic $\{0\}\times(-L,L)$ in~$\Omega_\infty$.
\end{proposition}

\begin{proof} 
Since $\alpha_k\nearrow+\infty$ as $k\to+\infty$, 
it follows from Proposition~\ref{thm:finiteintersectionbounded} that the length of $\gamma_k$ is bounded.
For all $\ell$, up to subsequences, $\widehat\gamma_k$ must then converge to an affinely parametrized geodesic segment in $\Omega_\infty$ (possibly with multiplicity), by Proposition~\ref{thm:limitstrip}. 
Since $\gamma_k$ is orthogonal to $\partial\Omega_{\alpha_k}$ at both endpoints, and $\alpha_k\to\infty$, the function $\eta(x_{4,k})^2x_{1,k}'$ converges to $0$ uniformly, which implies that the limit of $\widehat\gamma_k$ is a segment of vertical geodesic in $\Omega_\infty$, and that $\Delta_{\gamma_k}\searrow0$ as $k\to+\infty$.
This vertical geodesic is $\{0\}\times(-L,L)$, by Corollary~\ref{thm:finiteintersectionbounded} and the fact that $\gamma_k$ intersects $\gamma_\hor$ exactly $m$ times.
For $k$ sufficiently large, each interval $I_\ell$ contains exactly one zero of the function $x_{4,k}$, because $\widehat\gamma_k$ is $C^1$-close to a geodesic in $\Omega_\infty$. Such limit geodesic intersects the horizontal geodesic $m$ times, and these intersections correspond to instants uniformly distributed in the interval $(-1,1)$, since translations in $x$ are isometries of $\Omega_\infty$. 
By the same argument, the endpoints of $\widehat\gamma_k\vert_{I_\ell}$ correspond to sharp turns of $\widehat\gamma_k$ that converge to the endpoints of the vertical geodesic $\{0\}\times(-L,L)$. 
\end{proof}

We are now finally ready to prove Theorem~\ref{mainthmB} in the Introduction.

\begin{proof}[Proof of Theorem~\ref{mainthmB}]
Let $m\geq2$ be fixed. Let $a_m$ be as in \eqref{eq:am},
and $(a(t),s(t))$ be the curve \eqref{eq:bif-curve-extended} as in Theorem~\ref{thm:rabinowitz} corresponding to the bifurcation branch $\mathcal B_{a_*}$ with $a_*=a_m$.
For each $a>a_m$, let $t_{\mathrm{min}}(a):=\min\{ t\geq 0 : a(t)=a\}$ and set 
\begin{equation}
S_m(a):=\Pi^{-1}\big(\gamma_{a,s(t_{\mathrm{min}}(a))}\big),
\end{equation}
which is a nonplanar embedded minimal $2$-sphere in $E(a,b,b,d)$, by Proposition~\ref{prop:even_odd}.

By the proof of Theorem~\ref{thm:mainthm-bifbranches}, namely the fact that \eqref{eq:z} is constant along each bifurcating branch, it follows that $\gamma_{a,s(t_{\mathrm{min}}(a))}$ intersects $\gamma_\hor$ exactly at $m$ points. Therefore, $S_m(a)$ intersects $\Sigma_4(a)$ exactly along $m$ circles.
From Proposition~\ref{prop:corgammaan-new}, the geodesic $\gamma_{a,s(t_{\mathrm{min}}(a))}$ converges smoothly to $\{0\}\times (-L,L)$ with multiplicity $m$ as $a\nearrow+\infty$, and hence $S_m(a)$ converges smoothly to $\Sigma_1(\infty)$ with multiplicity $m$, away from the points $(0,0,0,\pm d)\in\R^4$, which correspond to $(0,\pm L)\in\partial\Omega_\infty$.
In particular, their areas also converge: $|S_m(a)|\to m|\Sigma_1(\infty)|$ as $a\nearrow+\infty$.

The Morse index of $S_m(a)$ is greater than or equal to its equivariant Morse index, which coincides with the Morse index of the geodesic $\gamma_{a,s(t_{\mathrm{min}}(a))}$. 
For $a$ sufficiently large, this index is at least $m-1$ by Proposition~\ref{prop:corgammaan-new}, since this geodesic develops $m-1$ sharp turns near $\partial \Omega_a$, and variations pushing these portions of the geodesic closer to $\partial\Omega_a$ decrease its length.
(These sharp turns correspond to catenoidal necks on $S_m(a)$, cf.~the structural results in \cite{cm-V}.)
Inequality \eqref{eq:liminf-scarring} follows readily. 
\end{proof}

\begin{appendix}
\section{\texorpdfstring{Arithmetic equations satisfied by $a_m$}{Arithmetic equations satisfied by bifurcation instants}}\label{appendixA}

In this Appendix, we use Heun functions to write the general solution to the singular ODE $\mathcal L_a(v)=0$ and provide arithmetic equations involving infinite continued fractions that are satisfied by the bifurcation instants $a_m$ in Theorem~\ref{mainthmB}.
Recall that $(a_m)_{m\geq1}$, defined in \eqref{eq:am}, alternates between the sequences 
$(a_n^\even)_{n\geq1}$ and $(a_n^\odd)_{n\geq0}$ of values of $a>0$ such that $\mathcal L_a(v)=0$ has a nontrivial solution in $V_\even$ and $V_\odd$, respectively, cf.~Propositions~\ref{prop:char-degeneracies} and \ref{prop:degsequences}, and \eqref{eq:vevenvodd}.

\subsection{Heun functions}
The \emph{(local) Heun function} $H\ell(\zeta,q;\alpha,\beta,\gamma,\delta;z)$ is defined as the unique solution $w(z)$ holomorphic at $z=0$ to the singular ODE
\begin{equation}\label{eq:heun}
\begin{cases}
\displaystyle\frac{\dd^2 w}{\dd z^2}+\left(\frac{\gamma}{z}+\frac{\delta}{z-1}+\frac{\epsilon}{z-\zeta}\right)\frac{\dd w}{\dd z}+\frac{\alpha\beta z-q}{z(z-1)(z-\zeta)}w=0,\\[10pt]
w(0)=1, \quad 
\displaystyle w'(0)=\frac{q}{\gamma \,\zeta},
\end{cases}
\end{equation}
where $\epsilon=\alpha+\beta-\gamma-\delta+1$, see e.g.~\cite{heun,sk-nist}.
(The above parameter $\zeta$ is typically denoted $a$ in the literature, but we shall use $\zeta$ here to distinguish it from the parameter $a>0$ used throughout the paper.) Equation \eqref{eq:heun} has four regular singular points, namely $0$, $1$, $\zeta$, and $\infty$, with exponents $(0,1-\gamma)$, $(0,1-\delta)$, $(0,1-\epsilon)$, and $(\alpha,\beta)$, respectively. 
The following can be found in \cite[\S 31.4]{sk-nist}:

\begin{lemma}\label{lemma:Heun}
There is an infinite sequence $(q_m)_{m\geq0}$ of values for the auxiliary parameter $q$ such that $H\ell(\zeta,q;\alpha,\beta,\gamma,\delta;z)$ is analytic at $z=1$, given by solutions~to 
\begin{equation}\label{eq:contfrac}
q=\dfrac{\zeta\, \gamma \, P_1}{Q_1+q-\dfrac{R_1\,P_2}{Q_2+q-\dfrac{R_2\, P_3}{Q_3+q- {}_{\ddots}}}}
\end{equation}
where the sequences $(P_j)_{j\in\N}$, $(Q_j)_{j\in\N}$, and $(R_j)_{j\in\N}$ are given by:
\begin{align*}
P_j &:= \big(j-1+\alpha\big)\big(j-1+\beta\big), \\
\phantom{, \qquad\quad j\in\N.}Q_j &:= j\,\Big(\big(j-1+\gamma\big)\big(1+\zeta \big)+\zeta \delta +\epsilon \Big), \qquad\quad j\in\N.\\
R_j &:= \zeta(j+1)\big(j+\gamma\big),
\end{align*}
\end{lemma}

\subsection{\texorpdfstring{Solutions to $\mathcal L_a(v)=0$}{Solutions to the singular ODE}}
For all $a>0$, the extension of $\mathcal L_a$ to a differential operator on the extended complex plane $\widehat \C=\C\cup\{\infty\}$ has regular singular points at $z=\pm1$ and $z=\pm\zeta$, where $\zeta:=\frac{a^2}{a^2-b^2}$.
Performing a change of variables with the linear fractional transformation that moves these regular singular points to $\{0,1,\zeta,\infty\}$, we see that $v(z)$ solves $\mathcal L_a(v)=0$ if and only if
\begin{equation*}\label{eq:gensolv}
v(z)= C_\even\,v_\even(z)+C_\odd \,v_\odd(z),
\end{equation*}
for some $C_\even,C_\odd\in\R$, where $v_\even(z)$ and $v_\odd(z)$ are given by
\begin{equation}\label{eq:vevenvoddHeun}
\begin{aligned}
v_\even(z) &= H\ell\left(\zeta,q_\even;\alpha_\even, \beta_\even,\tfrac12,d^2;z^2\right)E(z)\\
v_\odd(z) &= H\ell\left(\zeta,q_\odd;\alpha_\odd, \beta_\odd,\tfrac32,d^2;z^2\right)E(z)\,z,
\end{aligned}
\end{equation}
where $E(z):=\Big(1+\big(\tfrac{b^2}{a^2}-1\big)z^2\Big)^{1+\frac{d^2}{2}}$, and the parameters $q,\alpha,\beta$ are given by: 
\begin{equation}\label{eq:heun-pars}
\begin{aligned}
q_\even&=\tfrac{a^2 \left(d^2-b^2+2\right)-b^2 \left(d^2+2\right)-a^4}{4 \left(a^2-b^2\right)}, &
q_\odd&=\tfrac{a^2 \left(4 d^2-b^2+6\right)-2 b^2 \left(d^2+3\right)-a^4}{4 (a^2- b^2)},\\
\alpha_\even&=\tfrac{3 d^2+3+\sqrt{4 a^2+\left(d^2-1\right)^2}}{4}, &
\alpha_\odd&=\tfrac{3 d^2+5+\sqrt{4 a^2+\left(d^2-1\right)^2}}{4},  \\
\beta_\even&=\tfrac{3 d^2+3-\sqrt{4 a^2+\left(d^2-1\right)^2}}{4}, &
\beta_\odd&=\tfrac{3 d^2+5-\sqrt{4 a^2+\left(d^2-1\right)^2}}{4}. 
\end{aligned}
\end{equation}
Note that $v_\even(0)=1$ and $v_\even'(0)=0$, while $v_\odd(0)=0$ and $v_\odd'(0)=1$.
One may check that $v_\even\in V_\even$ if and only if $H\ell\left(\zeta,q_\even;\alpha_\even, \beta_\even,\tfrac12,d^2;z^2\right)$ is analytic at $z=1$,  
and $v_\odd\in V_\odd$ if and only if $H\ell\left(\zeta,q_\odd;\alpha_\odd, \beta_\odd,\tfrac32,d^2;z^2\right)$ is analytic at $z=1$. Therefore, from Lemma~\ref{lemma:Heun}, we obtain the following:

\begin{proposition}\label{prop:contfraceqns}
The sequences $(a_n^\even)_{n\geq1}$ and $(a_n^\odd)_{n\geq0}$ are among the values of $a>0$ such that 
the corresponding $q_\even$ and $q_\odd$, respectively, solve equation \eqref{eq:contfrac} in Lemma~\ref{lemma:Heun}, with the parameters $\alpha$ and $\beta$ replaced as in \eqref{eq:heun-pars}; the parameter $\gamma$ given by $\tfrac12$ in the even case, and $\tfrac32$ in the odd case; $\delta=d^2$; and $\zeta=\frac{a^2}{a^2-b^2}$.
\end{proposition}

\begin{remark}
Note that if $a=b$, then the above change of variables moves the regular singular points to $\{0,1,\infty,\infty\}$, the Heun functions in \eqref{eq:vevenvoddHeun} become so-called \emph{confluent} Heun functions, and $E(z)\equiv1$. In particular, if $a=b=c=d$, then \eqref{eq:vevenvoddHeun} simplify to $v_\even(z)=1-z\,\mathrm{arctanh}(z)$ and $v_\odd(z)=z$, cf.~Lemma~\ref{rem:unbounded-solutions}.
\end{remark}

Finding the values of $q$ that solve a continued fraction equation such as \eqref{eq:contfrac} 
is a very difficult arithmetic problem. Nevertheless, numerical experiments with \eqref{eq:vevenvoddHeun} suggest that, in the special case $b=d=1$, it might be reasonable to conjecture that $a_n^\even=2n$ and $a_n^\odd=2n+1$, for all $n\in\N$, i.e., $a_m=m$ for all $m\geq1$.

\end{appendix}

\providecommand{\bysame}{\leavevmode\hbox to3em{\hrulefill}\thinspace}
\providecommand{\MR}{\relax\ifhmode\unskip\space\fi MR }
\providecommand{\MRhref}[2]{%
  \href{http://www.ams.org/mathscinet-getitem?mr=#1}{#2}
}
\providecommand{\href}[2]{#2}


\begin{thebibliography}{BdlPJBP21}

\bibitem[AB15]{mybook}
{\sc M.~M. Alexandrino and R.~G. Bettiol}, \emph{Lie groups and geometric
  aspects of isometric actions}, Springer, Cham, 2015. \MR{3362465}

\bibitem[Alm66]{almgren}
{\sc F.~J. Almgren, Jr.}, \emph{Some interior regularity theorems for minimal
  surfaces and an extension of {B}ernstein's theorem}, Ann. of Math. (2)
  \textbf{84} (1966), 277--292. \MR{0200816}


\bibitem[BK19]{bamler-kleiner}
{\sc R.~H. Bamler and B.~Kleiner}, \emph{Ricci flow and contractibility of
  spaces of metrics}, arXiv:1909.08710.

\bibitem[BdlPJBP21]{petean}
{\sc A.~Betancourt de~la Parra, J.~Julio-Batalla, and J.~Petean}, \emph{Global
  bifurcation techniques for {Y}amabe type equations on {R}iemannian
  manifolds}, Nonlinear Anal. \textbf{202} (2021), 112140. \MR{4151198}

\bibitem[BP20]{bp-notices}
{\sc R.~G. Bettiol and P.~Piccione}, \emph{Instability and bifurcation},
  Notices Amer. Math. Soc. \textbf{67} (2020), no.~11, 1679--1691. \MR{4201907}

\bibitem[BP22]{spjms}
{\sc \bysame}, \emph{Global bifurcation for a class of nonlinear {ODE}s},
  S\~{a}o Paulo J. Math. Sci. \textbf{16} (2022), no.~1, 486--507. \MR{4426405}

\bibitem[BT03]{toland}
{\sc B.~Buffoni and J.~Toland}, \emph{Analytic theory of global bifurcation},
  Princeton Series in Applied Mathematics, Princeton University Press,
  Princeton, NJ, 2003. \MR{1956130}

\bibitem[CS85]{choi-schoen}
{\sc H.~I. Choi and R.~Schoen}, \emph{The space of minimal embeddings of a
  surface into a three-dimensional manifold of positive {R}icci curvature},
  Invent. Math. \textbf{81} (1985), no.~3, 387--394. \MR{807063}


\bibitem[CDL05]{cdl}
{\sc T.~H. Colding and C.~De~Lellis}, \emph{Singular limit laminations, {M}orse
  index, and positive scalar curvature}, Topology \textbf{44} (2005), no.~1,
  25--45. \MR{2103999}

\bibitem[CM15]{cm-V}
{\sc T.~H. Colding and W.~P. Minicozzi, II}, \emph{The space of embedded
  minimal surfaces of fixed genus in a 3-manifold {V}; fixed genus}, Ann. of
  Math. (2) \textbf{181} (2015), no.~1, 1--153. \MR{3272923}

\bibitem[CR71]{crandall-rabinowitz}
{\sc M.~G. Crandall and P.~H. Rabinowitz}, \emph{Bifurcation from simple
  eigenvalues}, J. Functional Analysis \textbf{8} (1971), 321--340.
  \MR{0288640}


\bibitem[EGNT13]{egnt}
{\sc J.~Eckhardt, F.~Gesztesy, R.~Nichols, and G.~Teschl},
  \emph{Weyl-{T}itchmarsh theory for {S}turm-{L}iouville operators with
  distributional potentials}, Opuscula Math. \textbf{33} (2013), no.~3,
  467--563. \MR{3046408}

\bibitem[HK19]{hk}
{\sc R.~Haslhofer and D.~Ketover}, \emph{Minimal 2-spheres in 3-spheres}, Duke
  Math. J. \textbf{168} (2019), no.~10, 1929--1975. \MR{3983295}

\bibitem[HNR03]{hnr}
{\sc J.~Hass, P.~Norbury, and J.~H. Rubinstein}, \emph{Minimal spheres of
  arbitrarily high {M}orse index}, Comm. Anal. Geom. \textbf{11} (2003), no.~3,
  425--439. \MR{2015753}


\bibitem[Hat83]{hatcher}
{\sc A.~E. Hatcher}, \emph{A proof of the {S}male conjecture, {${\rm
  Diff}(S^{3})\simeq {\rm O}(4)$}}, Ann. of Math. (2) \textbf{117} (1983),
  no.~3, 553--607. \MR{701256}

\bibitem[Hsi83a]{hsiang2}
{\sc W.-Y. Hsiang}, \emph{Minimal cones and the spherical {B}ernstein problem.
  {I}}, Ann. of Math. (2) \textbf{118} (1983), no.~1, 61--73. \MR{707161}

\bibitem[Hsi83b]{hsiang1}
{\sc \bysame}, \emph{Minimal cones and the spherical {B}ernstein problem.
  {II}}, Invent. Math. \textbf{74} (1983), no.~3, 351--369. \MR{724010}

\bibitem[HH82]{hsiang-hsiang}
{\sc W.-T. Hsiang and W.-Y. Hsiang}, \emph{On the existence of codimension-one
  minimal spheres in compact symmetric spaces of rank {$2$}. {II}}, J.
  Differential Geometry \textbf{17} (1982), no.~4, 583--594 (1983). \MR{683166}

\bibitem[HL71]{hsiang-lawson}
{\sc W.-Y. Hsiang and H.~B. Lawson, Jr.}, \emph{Minimal submanifolds of low
  cohomogeneity}, J. Differential Geometry \textbf{5} (1971), 1--38.
  \MR{298593}

\bibitem[JLX08]{yanyan}
{\sc Q.~Jin, Y.~Li, and H.~Xu}, \emph{Symmetry and asymmetry: the method of
  moving spheres}, Adv. Differential Equations \textbf{13} (2008), no.~7-8,
  601--640. \MR{2479025}

\bibitem[Kie04]{kielhofer}
{\sc H.~Kielh\"{o}fer}, \emph{Bifurcation theory}, Applied Mathematical
  Sciences, vol. 156, Springer-Verlag, New York, 2004, An introduction with
  applications to PDEs. \MR{2004250}

\bibitem[Rab71]{rabinowitz}
{\sc P.~H. Rabinowitz}, \emph{Some global results for nonlinear eigenvalue
  problems}, J. Functional Analysis \textbf{7} (1971), 487--513. \MR{0301587}

\bibitem[Ron95]{heun}
{\sc A.~Ronveaux (ed.)}, \emph{Heun's differential equations}, Oxford Science
  Publications, The Clarendon Press, Oxford University Press, New York, 1995,
  With contributions by F. M. Arscott, S. Yu. Slavyanov, D. Schmidt, G. Wolf,
  P. Maroni and A. Duval. \MR{1392976}

\bibitem[Ros95]{Ros95}
{\sc A.~Ros}, \emph{A two-piece property for compact minimal surfaces in a
  three-sphere}, Indiana Univ. Math. J. \textbf{44} (1995), no.~3, 841--849.
  \MR{1375352}

\bibitem[SK10]{sk-nist}
{\sc B.~D. Sleeman and V.~B. Kuznetsov}, \emph{Heun functions}, N{IST} handbook
  of mathematical functions, U.S. Dept. Commerce, Washington, DC, 2010,
  pp.~709--721. \MR{2655371}

\bibitem[SZ21]{scarring}
{\sc A.~Song and X.~Zhou}, \emph{Generic scarring for minimal hypersurfaces
  along stable hypersurfaces}, Geom. Funct. Anal. \textbf{31} (2021), 948--980.

\bibitem[Tes12]{teschl}
{\sc G.~Teschl}, \emph{Ordinary differential equations and dynamical systems},
  Graduate Studies in Mathematics, vol. 140, American Mathematical Society,
  Providence, RI, 2012. \MR{2961944}

\bibitem[Wei87]{weidmann}
{\sc J.~Weidmann}, \emph{Spectral theory of ordinary differential operators},
  Lecture Notes in Mathematics, vol. 1258, Springer-Verlag, Berlin, 1987.
  \MR{923320}

\bibitem[Whi87]{white87}
{\sc B.~White}, \emph{The space of {$m$}-dimensional surfaces that are
  stationary for a parametric elliptic functional}, Indiana Univ. Math. J.
  \textbf{36} (1987), no.~3, 567--602. \MR{905611}

\bibitem[Whi91]{white2}
{\sc \bysame}, \emph{The space of minimal submanifolds for varying {R}iemannian
  metrics}, Indiana Univ. Math. J. \textbf{40} (1991), no.~1, 161--200.
  \MR{1101226}

\bibitem[Whi16]{white-notes}
{\sc \bysame}, \emph{Introduction to minimal surface theory}, Geometric
  analysis, IAS/Park City Math. Ser., vol.~22, Amer. Math. Soc., Providence,
  RI, 2016, pp.~387--438. \MR{3524221}

\bibitem[Yau87]{yau-prob}
{\sc S.-T. Yau}, \emph{Nonlinear analysis in geometry}, Enseign. Math. (2)
  \textbf{33} (1987), no.~1-2, 109--158. \MR{896385}

\bibitem[Zet05]{zettl}
{\sc A.~Zettl}, \emph{Sturm-{L}iouville theory}, Mathematical Surveys and
  Monographs, vol. 121, American Mathematical Society, Providence, RI, 2005.
  \MR{2170950}

\bibitem[WZ]{wangzhou}
{\sc Z.~Wang and X.~Zhou}, \emph{Existence of four minimal spheres in $S^3$ with a bumpy metric}, arXiv:2305.08755.
  \MR{2170950}


\end{thebibliography}
\end{document}